\documentclass[preprint, 11pt]{elsarticle}

\usepackage[T1]{fontenc}
\usepackage[utf8]{inputenc}
\usepackage[hmargin=0.7in, vmargin=0.8in]{geometry}
\usepackage[babel,german=quotes]{csquotes}
\usepackage{subfig}
\usepackage{graphicx}
\usepackage[colorlinks=true,citecolor=blue,urlcolor=blue]{hyperref}
\usepackage{caption}
\usepackage{amsmath}          
\usepackage{amsthm}           
\usepackage{amssymb}          
\usepackage{bm, amsfonts}     
\usepackage{xcolor}
\usepackage{fixmath}
\usepackage{tikz}
\usepackage{bm}
\usepackage{makecell}
\usepackage{mathdots}
\usepackage{algpseudocode}
\usepackage{algorithm}
\newtheorem{thm}{Theorem}
\newdefinition{cor}{Corollary}

\newdefinition{example}{Example}
\newproof{pf}{Proof}

\usepackage[makeroom]{cancel}
\usepackage{tikz}

\newcommand {\Dt} {\Delta t}
\newcommand {\dd} {\text{d}}

   \newcommand{\bff}{\mathbf{f}}

\newcommand{\uA}{\bar u^A}
\newcommand{\uD}{\bar u^D}
\newcommand{\lambdaS}{a}

  \newcommand{\R}{\mathbb{R}}
  \newdefinition{rmk}{Remark}
  
  \newcommand{\pd}[2]{\frac{\partial #1}{\partial #2}}

\newcommand{\beq}{\begin{equation}}
\newcommand{\eeq}{\end{equation}}

\newcommand{\bfn}{{\bf n}}

\newcommand{\bfx}{{\bf x}}

\newcommand {\jinN} {j\in\mathcal N_i}
\newcommand {\N} {\mathcal N}
\newcommand{\RK}{{\rm RK}}
\newcommand{\FCT}{{\rm FCT}}

\newcommand{\GMC}{{\rm GMC}}

\makeatletter
\def\ps@pprintTitle{%
  \let\@oddhead\@empty
  \let\@evenhead\@empty
  \def\@oddfoot{
    \footnotesize\itshape
    \hfill\today
  }%
  \let\@evenfoot\@oddfoot}
\makeatother

\begin{document}

\begin{frontmatter} 

\title{
  Maximum principle preserving space and time flux limiting for
  Diagonally Implicit Runge-Kutta
  discretizations of scalar convection-diffusion equations}
  
\author[KAUST]{Manuel Quezada de Luna\corref{cor1}}
\ead{manuel.quezada@kaust.edu.sa}
\cortext[cor1]{Corresponding author}

\author[KAUST]{David I. Ketcheson}
\ead{david.ketcheson@kaust.edu.sa}

\address[KAUST]{King Abdullah University of Science and Technology (KAUST)\\ Thuwal 23955-6900, Saudi Arabia}

\journal{Elsevier}

\begin{abstract}
    We provide a framework for high-order discretizations of nonlinear
    scalar convection-diffusion equations that satisfy a discrete maximum
    principle.  The resulting schemes can have arbitrarily high order accuracy in time
    and space, and can be stable and maximum-principle-preserving (MPP)
    with no step size restriction.  The schemes are based on a two-tiered limiting
    strategy, starting with a high-order limiter-based method that may have
    small oscillations or maximum-principle violations, followed by an
    additional limiting step that removes these violations while preserving
    high order accuracy.  The desirable properties of the resulting schemes 
    are demonstrated through several numerical examples.
\end{abstract}

\begin{keyword}
  scalar convection-diffusion equations;
  positivity-preserving implicit schemes; 
  diagonally implicit Runge-Kutta time stepping; 
  flux-corrected transport; 
  monolithic convex limiting
\end{keyword}
\end{frontmatter}


\section{Introduction}
In this work we develop numerical methods for the
scalar nonlinear convection-diffusion equation
\beq
\pd{u}{t}+\nabla\cdot\mathbf{f}(u)-\nabla \cdot (c(u,\bfx) \nabla u)=0,
\qquad \bfx \in \Omega \subset\R^d,\quad d\in\{1,2,3\}, \label{conslaw}
\eeq
with $c(u,\bfx)\geq 0$.  We focus on methods that work well for regimes
ranging from the hyperbolic setting ($c=0$) to the diffusion-dominated setting.

A great deal of research has been devoted to the development of
numerical methods for hyperbolic conservation laws that are accurate
and preserve qualitative solution properties, such as guaranteeing
\emph{a priori} bounds on the solution or avoiding spurious oscillations.
Among the most useful schemes available are high-order finite volume
methods, which tend to exhibit much less oscillation than generic schemes,
but can still produce small over- or undershoots near discontinuities.
For some applications, these violations of the bounds are unacceptable
and more robust schemes are thus required. 

Certain broad theoretical results establish the difficulty of
guaranteeing properties like positivity or a maximum principle.
These often involve a tradeoff between accuracy and robustness.
Godunov's theorem dictates that any linear PDE discretization that does not generate
new local extrema must either be nonlinear or only first-order accurate.
Similarly, landmark results by Bolley \& Crouzeix \cite{bolley1978} and by Spijker \cite{spijker1983}
establish that any general linear method for initial value ODEs that is guaranteed
to satisfy a positivity or monotonicity property can be only first-order accurate.

\subsection*{High-order bound-preserving discretizations}
High-order discretizations that satisfy positivity, monotonicity, or 
total variation diminishing (TVD) properties must therefore fall outside the 
usual straightforward discretizations.
In the context of hyperbolic PDEs, such methods have constituted a major area of research
for several decades.  Well established techniques include second-order methods
based on TVD limiters, as well as flux corrected transport (FCT) methods.
Each of these imposes a local bound on solution updates, but such
methods
are at most second order accurate \cite{osher1984high, zhang2010maximum}, due to
being too restrictive around smooth extrema. 

Some approaches, like weighted essentially non-oscillatory (WENO) methods,
give up on providing strict preservation of the maximum principle
in favor of achieving better than second-order accuracy. 
Within the context of finite elements, high-order Bernstein polynomials can be employed with 
FCT-like methods to impose local bounds; see for example 
\cite{anderson2017high,hajduk2021monolithic,lohmann2017flux}.
To recover high-order accuracy in \cite{lohmann2017flux}, the authors 
use smoothness indicators to deactivate the limiters around smooth extrema. 
Therefore, small violations of the maximum principle could occur. 

High order accuracy can be achieved by instead imposing the global bounds 
\begin{align}\label{global_bounds}
  \min_j u^n_j \leq u_i^{n+1} \leq \max_j u^n_j
\end{align}
Note that \eqref{global_bounds} is essentially a one-step discrete form of
the maximum principle, which states that a solution must remain bounded by its maximum 
(and/or minimum) value at the beginning of the simulation. 
We will refer to schemes that satisfy \eqref{global_bounds} as
\emph{maximum principle preserving} (MPP).
Higher than second-order full discretizations that strictly satisfy \eqref{global_bounds} are a
more recent development and include \cite{zhang2011maximum,xiong2015high,chen2016third}.

In the present work, our starting point is a spatial discretization, like WENO
or the method presented in \cite{lohmann2017flux} for finite elements, that
uses limiters to achieve a considerable reduction of oscillations without
degrading the formal order of accuracy of the solution (but may violate
\eqref{global_bounds}).  We modify this discretization by performing an extra
flux limiting step that enforces \eqref{global_bounds} strictly.
Thus the resulting method employs a two-tier limiting strategy.

\subsection*{Bound-preserving time discretization}
Most of the works cited above are based on method-of-lines finite
volume or (dis)continuous Galerkin finite element discretizations.
A key difficulty in this area is to find a time integration
scheme that preserves the boundedness properties of the semi-discrete
scheme.  This difficulty is commonly solved by applying a strong stability preserving 
(SSP) Runge-Kutta time discretization.  This
means that the schemes are limited to 4th- or 6th-order accuracy,
depending on whether an explicit or implicit method is used \cite{SSPbook}.  Furthermore,
existing high-order implicit SSP methods are not A-stable, so that such
schemes will (when applied to \eqref{conslaw}) be subject to severe time step
restrictions even if an implicit time discretization is used.  

In \cite{arbogast2020third}, the authors take a different approach by combining
the backward Euler method with a third-order
fully-implicit Runge-Kutta method in order to have L-stability.  
However, the spatial discretization is based on WENO reconstruction,
which leads to a scheme that does not strictly satisfy the maximum principle. 
Herein, we provide a general technique that allows the use of any high-order 
Runge-Kutta method with a spatial discretization based on WENO reconstruction. 
The time discretization need not be SSP, can be of arbitrarily high-order, 
and can be chosen to be diagonally implicit.
To obtain a full discretization that satisfies the maximum principle, we combine 
the high-order method with a low-order MPP scheme based on 
backward Euler with local Lax-Friedrichs numerical fluxes. 
Therefore, we obtain anti-diffusive fluxes (or flux corrections) that contain corrections 
to the spatial and temporal components of the low-order scheme. 
The FCT method has been used before with fluxes that combine corrections in space and time; 
see for instance \cite{lee2010multistep,lohmann2017flux,xiong2015high,yang2016high}.
However, those references are based on combining explicit schemes.
As a result, their flux limited update is MPP only under a restricted time step. 
It is also worth mentioning \cite{feng2019time}, wherein the authors use continuous 
Galerkin finite elements in space and discontinuous Galerkin finite elements in time to 
obtain a full discretization that is then modified via FCT to obtain an MPP method. 
The baseline discretization in \cite{feng2019time} resembles an implicit scheme.

\subsection*{Schemes for problems with parabolic terms}
Much work in this area has focused on the purely hyperbolic setting,
since parabolic terms tend to have a smoothing effect and may make
it less challenging to achieve discrete boundedness properties.
Some recent works have focused on convection-diffusion applications
where convection is dominant and it is still difficult to
avoid oscillations or strictly satisfy the maximum principle
\cite{zhang2012maximum,chen2016third,xiong2015high,yang2016high}.
These methods are explicit, and will be subject to tight restrictions
on the time step when the diffusion coefficient is not small.
In addition, strict preservation of the maximum principle
is fulfilled only if a time step restriction is satisfied. 
In this work, we consider implicit schemes. As a result, we obtain a high-order 
full discretization that is MPP with no time step restriction. 

\subsection{Our contribution}
The techniques proposed here are closely related to those of \cite{exGMCL},
which provided fully discrete explicit schemes for hyperbolic problems.
Here we extend the approach to problems that include diffusion and
to implicit time integration.
We make use of two ideas based on general techniques that have
long been used in this area.  The first is that of combining a
low-order method that satisfies the desired property with a high-order
method that may not, in such a way that the high-order "correction" is
guaranteed not to break the property.  The second is reminiscent of
Harten's Theorem \cite{harten1974method,harten1997high}, and consists
of writing a scheme as a sum of updates, each of which is proportional
to the difference between the current state and a neighboring state.
By bounding the neighboring states and the proportions, one obtains
bounds on the updated solution.

The main contributions of our new schemes are: 
i) strict enforcement of the maximum principle \eqref{global_bounds}; 
ii) arbitrarily high-order accuracy in time and space; 
iii) application not only to hyperbolic problems but also to convection-diffusion equations; 
and iv) linear stability and MPP under arbitrarily large time step sizes.
The result is a framework to obtain an arbitrarily high-order full discretization for 
the convection-diffusion equation that is strictly MPP for time steps of any size. 

\subsection{Outline}
The rest of this manuscript is organized as follows. 
In Section \ref{sec:flux_correction}, we review a general and well-known framework for discretizations of 
convection-diffusion problems based on limiting flux corrections. This methodology relies on low- and high-order 
schemes, which we present in Sections \ref{sec:low_order} and \ref{sec:high_order}, respectively. 
Afterwards, in Sections \ref{sec:fct} and \ref{sec:gmc}, we use two flux limiting techniques 
that guarantee the scheme is MPP.
In Section \ref{sec:int_stages}, we discuss how to impose the MPP on the 
intermediate solutions within the stages of the Runge-Kutta method.
In Section \ref{sec:num}, we present four one-dimensional 
tests and four two-dimensional tests. 
In our numerical examples, we use a 5th-order WENO discretization and a
5th-order singly  diagonally implicit Runge-Kutta (SDIRK) method. 
Concluding remarks are given in Section \ref{sec:conclusions}. 
For completeness, in \ref{sec:jacobians}, we discuss how we solve the 
algebraic equations associated with the implicit discretizations.

\section{Flux correction}\label{sec:flux_correction}
We are interested in a finite volume spatial discretization
of the convection-diffusion equation \eqref{conslaw}.
For simplicity, we assume the domain  $\Omega$ is a hyperrectangle and prescribe
periodic boundary conditions on $\partial\Omega$. The initial condition is given by
\beq
u(\bfx, 0)=u_0(\bfx) \qquad\mbox{in}\ \Omega.
\eeq

We partition $\Omega$ into $N_h$ cells $K_i$, where $i=1,\dots N_h$.
Let $\partial K_i$ denote the boundary of $K_i$,
$S_{ij}=\partial K_i\cap\partial K_j$ denote the face shared between cells $K_i$ and $K_j$,
$|K_i|$ and $|S_{ij}|$ denote the volume and the area of $K_i$ and $S_{ij}$, respectively,
and $\N_i$ denote the set of indices of neighbors of $K_i$ that share a face with it. 
To obtain a finite volume semi-discretization,
we integrate \eqref{conslaw} over each cell, apply the divergence theorem
to the convective and diffusive terms, and approximate the fluxes via
\begin{align*}
  F_{ij}(u,\bfx,t) \approx \bff(u)\cdot\bfn_{ij},
  \qquad 
  P_{ij}(u,\bfx,t) \approx c(u,\bfx)(\nabla u\cdot\bfn_{ij}),
\end{align*}
on each face $S_{ij}$.
Here $\bfn_{ij}$ is the unit outward normal on face $S_{ij}$.
Doing so, we get the spatial semi-discretization 
\begin{align}\label{semi}
  |K_i|\frac{\dd u_i}{\dd t} &= - \sum_{\jinN}\int_{S_{ij}}\left[F_{ij}(u,\bfx,t)-P_{ij}(u,\bfx,t)\right]ds,
  \qquad i=1,\dots, N_h,
\end{align}
where $u_i$ is the average of the solution over cell $K_i$.
We consider full discretizations of the form
\begin{align} \label{flux-scheme}
  u^{n+1}_i & = u^n_i - \frac{\Dt}{|K_i|} \sum_{\jinN} \int_{S_{ij}}G_{ij}(\bfx) ds,
\end{align}
where 
\begin{align*}
  G_{ij}(\bfx)\approx\frac{1}{\Dt}\int_{t_n}^{t_{n+1}} (\mathbf{f}(u(\bfx,t)) - c(u(\bfx,t),\bfx)\nabla(u(\bfx,t)))  \cdot \bfn_{ij} dt
\end{align*}
is a time-averaged approximation of the combined flux $F_{ij} - P_{ij}$.
The specific form of $G_{ij}$ depends on the time integration scheme. In Sections 
\ref{sec:low_order} and \ref{sec:high_order}, we consider 
backward Euler and diagonally-implicit Runge-Kutta (DIRK) methods, respectively. 

The basic idea used in this work was proposed almost 40 years ago in the
\emph{hybrid schemes} of Harten \& Zwas \cite{harten1972self} and in the
flux-corrected transport (FCT) algorithm of Boris \& Book \cite{boris1973flux}.
It has been employed in countless other methods proposed since then, and is explained
neatly for instance in \cite[Section~16.2]{leveque1992numerical} and \cite{kuzmin2012flux}.
The idea is to define two different numerical fluxes $G^L_{ij}$ and $G^H_{ij}$, leading
to two schemes, each of the form \eqref{flux-scheme}.
The low-order flux $G^L_{ij}$ is inaccurate but yields an update \eqref{flux-scheme}
that satisfies a desired bound.  The high-order flux $G^H_{ij}$ is more accurate
but does not generally satisfy the bound.  We then apply the scheme
\begin{align} \label{limited-flux-scheme}
  u^{n+1}_i & = u^n_i - \frac{\Dt}{|K_i|} \sum_{\jinN} \int_{S_{ij}} 
  \left[G^L_{ij}(\bfx) - \alpha_{ij}(G^L_{ij}(\bfx)-G^H_{ij}(\bfx))\right] ds,
\end{align}
where the \emph{limiters} $\alpha_{ij}\in [0,1]$ depend on $u$ and are chosen to
maximize the accuracy while still enforcing the bound.
We refer to the difference $(G^L_{ij}-G^H_{ij})$ as the \emph{flux correction},
since it does not approximate the flux itself but rather provides a high-order
correction to it.
Herein, we approximate the fluxes $G_{ij}^{L/H}$ using the point value at the center 
of each face.

As described in Section \ref{sec:low_order}, using backward Euler and the
Lax-Friedrichs numerical flux yields low-order fluxes $G_{ij}^L$ that are MPP
with no restriction on the time step size. 
In Section \ref{sec:high_order}, we use arbitrarily high-order methods based on WENO
reconstruction and DIRK time integration to define $G_{ij}^H$. 
We present two approaches to choosing the values of the limiters; the
first is based on Zalesak's FCT limiters \cite{zalesak1979fully} (but applied in
time as well as space) while the second follows
the recently-proposed \emph{monolithic convex limiting} approach \cite{exGMCL}.

\section{Low-order scheme}\label{sec:low_order}
In this section, we define low-order convective and diffusive fluxes, which we denote by
$F_{ij}^L(u)\approx \bff(u)\cdot\bfn_{ij}$ and $P_{ij}^L(u)\approx (c(u,x)\nabla u)\cdot\bfn_{ij}$, respectively.

For the convective fluxes we use the local Lax-Friedrichs flux,
also known as Rusanov's flux, given by \cite{rusanov1961calculation}
\begin{align}\label{conv-flux_low_order}
  F^L_{ij}(u) =
    \bfn_{ij}\cdot \frac{\bff(u_j)+\bff(u_i)}{2}
    -\frac{1}{2}(u_j-u_i)\lambda^A_{ij},
\end{align}
where $u_j$ denotes the solution average over cell $K_j$ and $\lambda_{ij}^A>0$ is an
upper bound for the wave speed of the Riemann problem associated with face $S_{ij}$.
In general, $\lambda_{ij}^A$ is a function of $u$. For simplicity, we omit the dependence herein. 
For a uniform and structured grid, the diffusive flux can be taken as
\begin{align}\label{diff-flux_low_order}
  P_{ij}^L(u) =
    \underbrace{c\left(\frac{u_{i}+u_j}{2},\frac{\bfx_i+\bfx_j}{2}\right)}_{\qquad \textstyle =: c_{ij}}
    \frac{u_{j}-u_i}{|\bfx_j-\bfx_i|}.
\end{align}
For more general grids, we can follow e.g. \cite{nikitin2014monotone}. 
Plugging \eqref{conv-flux_low_order} and \eqref{diff-flux_low_order} into \eqref{semi}, we get
\begin{align}\label{low_order_semi1}
  |K_i|\frac{\dd u_i}{\dd t}
  = -\sum_{\jinN} |S_{ij}|
  \left[
    \bfn_{ij}\cdot \frac{\bff(u_j)+\bff(u_i)}{2}
    -\frac{1}{2}\lambda^A_{ij}(u_j-u_i)
    -c_{ij}\frac{u_{j}-u_i}{|\bfx_j-\bfx_i|}
    \right].
\end{align}
Note that the scheme is mass-conservative since the right hand side above is antisymmetric with
respect to the exchange of $i$ and $j$.
For the purely hyperbolic case ($c=0$) of \eqref{conslaw}, Guermond and Popov \cite{guermond2016invariant}
considered an equivalent representation of \eqref{low_order_semi1} in terms of upwinded averages 
\begin{align}\label{bar_state_A}
  \bar u_{ij}^A(u) := \frac{u_j+u_i}{2}
  -\bfn_{ij}\cdot \frac{\bff(u_j)-\bff(u_i)}{2\lambda_{ij}^A}.
\end{align}
We remark that these states appear naturally in certain approximate Riemann solvers,
where one assumes that the solution of the Riemann problem
with data $u_i$ and $u_j$ consists of two traveling discontinuities, as shown in Figure \ref{fig:bar_states}.
Then, $\bar u_{ij}^A$ is the intermediate state in the Riemann solution; see for instance \cite{leveque2002finite}.
These states satisfy \cite{guermond2016invariant, kuzmin2020monolithic}
\begin{align*}
  \min\{u_i,u_j\}\leq \bar u_{ij}^A(u) \leq \max\{u_i,u_j\}.
\end{align*}
We also define the arithmetic average states:
\begin{align}\label{bar_state_D}
  \min\{u_i,u_j\}\leq \bar u_{ij}^D(u) := \frac{u_j+u_i}{2} \leq \max\{u_i,u_j\}.
\end{align}

\begin{figure}[!h]
  \begin{center}
    \includegraphics[scale=0.5]{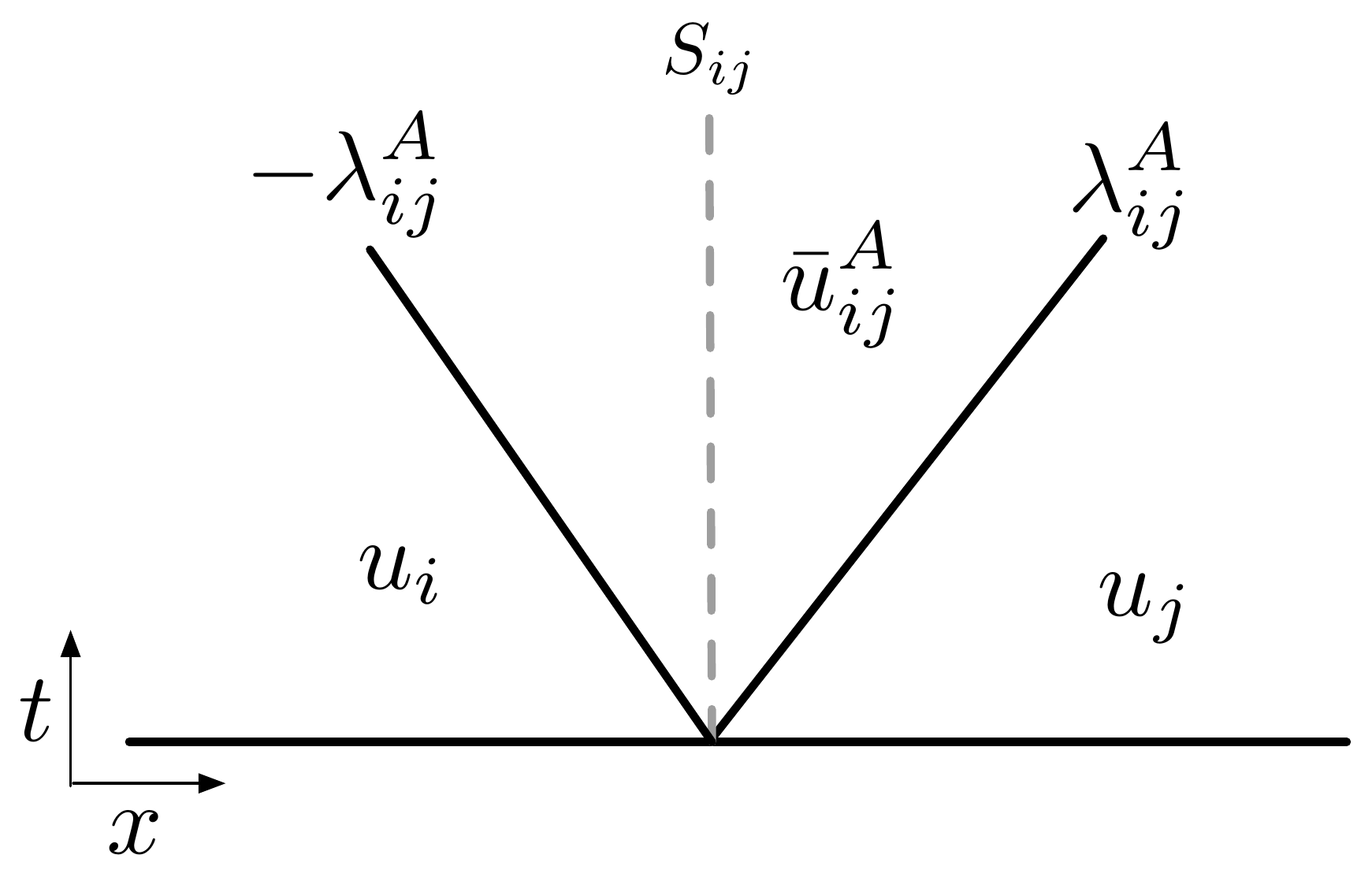}
    \caption{Structure of the Riemann solution using the local Lax-Friedrichs flux. 
      The left and right states are $u_i$ and $u_j$, the middle state is $\bar u_{ij}^A$.
      These states are separated by discontinuities traveling at speeds $\pm \lambda_{ij}^A$.
      \label{fig:bar_states}}
  \end{center}
\end{figure}

Using \eqref{bar_state_A}, \eqref{bar_state_D} and the fact that
$\sum_{\jinN}|S_{ij}|\bfn_{ij}\cdot \bff(u_i)=0$, we write \eqref{low_order_semi1} as follows:
\begin{align*}
  |K_i|\frac{\dd u_i}{\dd t}
  &= \sum_{\jinN} |S_{ij}| \lambda_{ij}^A
  \left[
    \frac{u_j-u_i}{2}
    -\bfn_{ij}\cdot \frac{\bff(u_j)-\bff(u_i)}{2\lambda_{ij}^A}
    +\frac{2c_{ij}}{\lambda_{ij}^A|\bfx_j-\bfx_i|}\frac{u_{j}-u_i}{2}
    \right] \\
  &= \sum_{\jinN} |S_{ij}| \lambda_{ij}^A
  \left[
    \frac{u_j+u_i}{2} - u_i
    -\bfn_{ij}\cdot \frac{\bff(u_j)-\bff(u_i)}{2\lambda_{ij}^A}
    +\frac{2c_{ij}}{\lambda_{ij}^A|\bfx_j-\bfx_i|}\left(\frac{u_{j}+u_i}{2}-u_i\right)
    \right] \\
  &= \sum_{\jinN} |S_{ij}| \lambda_{ij}^A
  \left[
    \bar u_{ij}^A - u_i
    +\frac{2c_{ij}}{\lambda_{ij}^A|\bfx_j-\bfx_i|}\left(\bar u_{ij}^D-u_i\right)
    \right] \\
  &= \sum_{\jinN}
  \left[
  |S_{ij}| \lambda_{ij}^A
  \left(
  \bar u_{ij}^A
  +\frac{2c_{ij}}{\lambda_{ij}^A|\bfx_j-\bfx_i|} \bar u_{ij}^D
  \right)
  -
  |S_{ij}| \lambda_{ij}^A
  \left(
  1
  +\frac{2c_{ij}}{\lambda_{ij}^A|\bfx_j-\bfx_i|}
  \right) u_i
  \right].
\end{align*}
We now introduce the following quantities:
\begin{align}\label{dbar_state}
  \lambda_{ij}   := \lambda^A_{ij}\left( 1+ \frac{2c_{ij}}{\lambda^A_{ij}|\bfx_j-\bfx_i|}\right), 
  \qquad
  \bar u_{ij}(u) := \frac{1}{1+\frac{2 c_{ij}}{\lambda^A_{ij}|\bfx_j-\bfx_i|}}
  \left(\uA_{ij}(u)+\frac{2 c_{ij}}{\lambda^A_{ij}|\bfx_j-\bfx_i|}\uD_{ij}(u)\right). 
\end{align}
Importantly, the states $\bar u_{ij}(u)$ are also bound preserving:
\begin{align*}
  \min\{u_i,u_j\} \leq \bar u_{ij} \leq \max\{u_i,u_j\}.
\end{align*}

Finally, we can write the low-order spatial semi-discretization \eqref{low_order_semi1} as follows:
\begin{align}\label{low_order_semi2}
  |K_i|\frac{\dd u_i}{\dd t} = \sum_{\jinN} |S_{ij}| \lambda_{ij}\left(\bar u_{ij}(u)-u_i \right).
\end{align}
Since $|K_i|, |S_{ij}|>0$, $\lambda_{ij}\geq 0$ and $\bar u_{ij}(u)\in[u_i,u_j]$, the low-order
spatial semi-discretization \eqref{low_order_semi2} is local extremum diminishing (LED)
\cite{jameson1993computational}.
Namely, if $u_i$ is a local maximum, $\dd u_i/\dd t \leq 0$ so $u_i$ can't increase. 
Similarly, if $u_i$ is a local minimum, $\dd u_i/\dd t \geq 0$ so $u_i$ can't decrease.
This leads to a semi-discretization that is MPP.  
By using the implicit Euler method, we achieve
the MPP property in time also, with no restriction on the step size; 
see for example \cite{horvath1998positivity,magiera2020constraint}. 
The full low-order scheme is thus
\begin{align} \label{low-order-scheme_bar_states}
  u_i^{L,n+1}=u_i^{L,n} + \frac{\Delta t}{|K_i|} \sum_{\jinN} |S_{ij}| \lambda_{ij}
  \left(\bar u_{ij}(u^{L,n+1})-u_i^{L,n+1} \right),
\end{align}
where we use the superscript $L$ to refer to the low-order solution based on the backward Euler scheme 
with Lax-Friedrichs numerical fluxes. 
We can write \eqref{low-order-scheme_bar_states} in terms of the low-order fluxes 
\begin{align}\label{low-order-fluxes}
  G_{ij}^L := F^L_{ij}\left(u^{L,n+1}\right) - P^L_{ij}\left(u^{L,n+1}\right),
\end{align}
as follows:
\begin{align} \label{low-order-scheme_fluxes}
  u^{L,n+1}_i & = u^{L,n}_i - \frac{\Delta t}{|K_i|} \sum_{\jinN} |S_{ij}| G^L_{ij}.
\end{align}

\begin{rmk}[Time step restriction with Forward Euler]
If we discretize \eqref{low_order_semi2} in time using the forward Euler method, we obtain
\begin{align*}
  u_i^{L,n+1} = \left(1-\frac{\Delta t}{|K_i|}\sum_{\jinN}|S_{ij}| \lambda_{ij} \right)u_i^{L,n} 
  + \frac{\Delta t}{|K_i|}\sum_{\jinN}|S_{ij}|\lambda_{ij} \bar u_{ij}(u^{L,n}).
\end{align*}
The solution is MPP provided 
\begin{align*}
  \left(1-\frac{\Delta t}{|K_i|}\sum_{\jinN} |S_{ij}|\lambda_{ij} \right)\geq 0
  \quad \implies \quad
  \Delta t \leq \frac{|K_i|}{\sum_{\jinN}|S_{ij}|\lambda_{ij}} \sim \frac{|K_i||\bfx_j-\bfx_i|}{|S_{ij}|} = \mathcal{O}(h^2),
\end{align*}
where $h$ is the mesh size.
\end{rmk}

\section{High-order scheme}\label{sec:high_order}
In scheme \eqref{limited-flux-scheme}, 
$G_{ij}^H$ are fluxes that improve the accuracy in space and time of the low-order fluxes $G_{ij}^L$. 
In this section, we define $G_{ij}^H$. 
Let us first define high-order convective and diffusive fluxes, which we denote by 
$F_{ij}^H(u)\approx \bff(u)\cdot\bfn_{ij}$ and $P_{ij}^H(u)\approx (c(u,x)\nabla u)\cdot\bfn_{ij}$, respectively.

For the high-order convective fluxes, we apply \eqref{conv-flux_low_order},
after replacing the cell averages by high-order pointwise reconstructed values.  
Let $p_i(u,\bfx)$ denote a high-order approximation of $u(\bfx)$ in cell $K_i$, 
based on weighted essentially non-oscillatory (WENO) reconstruction \cite{liu1994weighted, jiang1996efficient}.  
Then we set
\begin{align}\label{conv-flux_high_order}
  F^H_{ij}(u,\bfx)=
    \bfn_{ij}\cdot \frac{\bff(p_j(u,\bfx))+\bff(p_i(u,\bfx))}{2} 
    -\frac{1}{2}\lambda^A_{ij}(p_j(u,\bfx)-p_i(u,\bfx)).
\end{align}
For a uniform and structured mesh, the high-order diffusive fluxes can be given by 
\begin{align}\label{diff-flux_high_order}
  P_{ij}^H(u,\bfx) = 
    \frac{c(p_i(u,\bfx),\bfx)\nabla p_i(u,\bfx)+c(p_j(u,\bfx),\bfx)\nabla p_j(u,\bfx)}{2}\cdot\bfn_{ij}.
\end{align}
In principle, these fluxes should be integrated over each face $S_{ij}$, but the reconstruction
required for this quadrature is very expensive.  An economical alternative, which we use in Section \ref{sec:num},
is to approximate the spatial integrand by the value at the midpoint of the face; this approach often
reaps most of the benefits of the high-order WENO reconstruction at a reduced cost \cite{qiu2002construction}.
We correspondingly replace $\bfx$ by $\bfx_{ij}$ in \eqref{conv-flux_high_order} and \eqref{diff-flux_high_order},
where $\bfx_{ij}$ denotes the midpoint of face $S_{ij}$. 

To integrate in time, we use high-order $M$-stage diagonally implicit Runge-Kutta (DIRK) methods.
Let $b_m$, $c_m$, and $a_{ms}$ (with $m,s=1,\dots, M$) denote the Butcher coefficients of the DIRK method. 
The intermediate RK approximations $y_i^{(m)}\approx u_i(t^n+c_m\Delta t)$ to the cell averages are given by
\begin{align}\label{interm_RK_soln}
  y_i^{(m)} = u_i^n- \frac{\Delta t}{|K_i|}\sum_{\jinN}|S_{ij}| \sum_{s=1}^m a_{ms}
  \left[F^H_{ij}(y^{(s)},\bfx_{ij}) - P^H_{ij}(y^{(s)},\bfx_{ij})\right], 
  \qquad m=1,\dots, M.
\end{align}
The RK update is given by 
\begin{align} \label{high-order-scheme}
  u^{H,n+1}_i & = u^{H,n}_i - \frac{\Delta t}{|K_i|} \sum_{\jinN}|S_{ij}| G_{ij}^{H},
\end{align}
where 
\begin{align} \label{high-order-fluxes}
  G_{ij}^{H} = \sum_{m=1}^M b_m\left[F^H_{ij}\left(y^{(m)},\bfx_{ij}\right)-P^H_{ij}\left(y^{(m)},\bfx_{ij}\right)\right]
\end{align}
are the high-order fluxes. 
Here we use the superscript $H$ to refer to the high-order solution based on DIRK schemes with WENO reconstruction. 
We use the high-order flux $G_{ij}^{H}$ in scheme \eqref{limited-flux-scheme}. 

The high-order solution $u_i^H$ is not MPP due to violations introduced 
by the discretizations in space and time. Using scheme \eqref{limited-flux-scheme}, 
with the flux limiters that we introduce in the next two sections, guarantees the RK 
update is MPP. 
However, the intermediate solutions $y_i^{(m)}$ might still violate the maximum principle. 
For some applications, preservation of the maximum principle is also needed for the intermediate solutions;
we discuss ways to impose this in Section \ref{sec:int_stages}.

\begin{rmk}[Conservation of mass]
  From \eqref{conv-flux_high_order} and \eqref{diff-flux_high_order}, 
  $F_{ij}^H=-F_{ji}^H$ and $P_{ij}^H=-P_{ji}^H$; hence, 
  $G_{ij}^H=-G_{ji}^H \implies \sum_i |K_i|u_i^{H,n}=\sum_i |K_i|u_i^{H,0}$. Therefore, the scheme 
  \eqref{high-order-scheme} is mass conservative. 
\end{rmk}

\section{Flux corrected transport (FCT) limiting} \label{sec:fct}
Consider the scheme \eqref{limited-flux-scheme} with $G_{ij}^L$ and $G_{ij}^H$ given by 
\eqref{low-order-fluxes} and \eqref{high-order-fluxes}, respectively. 
Because these fluxes are constant on each face, we can write \eqref{limited-flux-scheme}
as
\begin{align}\label{lfs-cf}
  u_i^{n+1} 
  = u_i^n - \frac{\Delta t}{|K_i|}\sum_{\jinN} |S_{ij}|\left[G_{ij}^L - \alpha_{ij}(G_{ij}^L-G_{ij}^H)\right] 
  = u_i^{L,n+1} + \frac{\Delta t}{|K_i|} \sum_{\jinN} |S_{ij}|\alpha_{ij}(G_{ij}^L-G_{ij}^H).
\end{align}
In this section, we use the FCT method of \cite{boris1973flux,zalesak1979fully}
to determine the flux limiters $\alpha_{ij}$.
Although both the low- and high-order methods are implicit, and hence require
solving algebraic systems, the limiters are computed explicitly, as described below. 
The flux-limited update inherits the MPP properties of the 
low-order solution; that is, the solution is MPP with no time step restriction. 

In the rest of this section, we follow \cite{kuzmin2012flux}. 
We will determine flux limiters $\alpha_{ij}\in[0,1]$ that guarantee
\begin{align}\label{fct_bounds}
  \frac{|K_i|}{\Delta t}\left(u^{\min}-u_i^{L,n+1}\right) =:Q_i^-
  \leq
  \sum_{\jinN}|S_{ij}|\alpha_{ij}(G_{ij}^L-G_{ij}^H)
  \leq 
  Q_i^+:=\frac{|K_i|}{\Delta t}\left(u^{\max}-u_i^{L,n+1}\right).
\end{align}
Using condition \eqref{fct_bounds}, this guarantees
$u^{\min}\leq u_i^{n+1}\leq u^{\max}$.
The limiters are computed as follows: 
\begin{enumerate}\label{zalesak_limiters}
\item Calculate the sum of positive and negative flux corrections:
\begin{subequations}\label{limiters}
\begin{align}
  P_i^+ = \sum_{\jinN}|S_{ij}|\max\left\{0,G_{ij}^L-G_{ij}^H\right\},
  \qquad
  P_i^- = \sum_{\jinN}|S_{ij}|\min\left\{0,G_{ij}^L-G_{ij}^H\right\}.
\end{align}
\item Use the sums $P_i^{\pm}$ and the bounds $Q_i^{\pm}$, given by \eqref{fct_bounds}, to compute 
\begin{align}
  R_i^+ = \min\left\{1,\frac{Q_i^+}{P_i^+}\right\}, 
  \qquad
  R_i^- = \min\left\{1,\frac{Q_i^-}{P_i^-}\right\}.
\end{align}
\item Define the limiters by
\begin{align}
  \alpha_{ij} = 
  \begin{cases}
    \min\{R_i^+, R_j^-\}, &\mbox{ if } G_{ij}^L-G_{ij}^H \geq 0, \\
    \min\{R_i^-, R_j^+\}, &\mbox{ otherwise }.
  \end{cases}
\end{align}
\end{subequations}
\end{enumerate}
Clearly, $R_i^{\pm}\in[0,1]\implies \alpha_{ij}\in[0,1]$. 
The satisfaction of \eqref{fct_bounds} is proven as follows: 
\begin{align*}
  \sum_{\jinN}|S_{ij}|\alpha_{ij}(G_{ij}^L-G_{ij}^H) \leq
  \sum_{\jinN}|S_{ij}|\alpha_{ij}\max\{0,G_{ij}^L-G_{ij}^H\} \leq
  R_i^+ \sum_{\jinN} |S_{ij}| \max\{0, G_{ij}^L-G_{ij}^H\} \leq Q_i^+,
\end{align*}
and similarly for the lower bound $Q_i^-$. 
Since $G_{ij}^{L/H}=-G_{ji}^{L/H}$ and $\alpha_{ij}=\alpha_{ji}$,  we have
\begin{align*}
  \sum_i \sum_{\jinN} |S_{ij}|\alpha_{ij}(G_{ij}^L-G_{ij}^H)=0
  \implies
  \sum_{\jinN}|K_i|u_i^{n+1}=\sum_{\jinN}|K_i|u_i^{L,n+1},
\end{align*}
which, by conservation of $u_i^{L,n+1}$, implies the scheme \eqref{lfs-cf} is mass conservative. 

\begin{rmk}[Iterative FCT]
  In some of the numerical experiments from Section \ref{sec:num}, we use the iterative FCT method 
  to recover the high-order accuracy from the baseline scheme. 
  The basic idea behind iterative FCT is to consider the quantity
  $(1-\alpha_{ij})(G_{ij}^L-G_{ij}^H)$, which is the flux excluded by the limiters,
  and perform an extra limiting step given by
  \begin{align*}
    u_i^{n+1} = 
    u_i^{\FCT,n+1} + \frac{\Delta t}{|K_i|}\sum_{\jinN}|S_{ij}|\alpha_{ij}^{(2)}(1-\alpha_{ij})(G_{ij}^L-G_{ij}^H),
  \end{align*}
  where $u_i^{\FCT,n+1}$ is given by \eqref{lfs-cf} and the superscript $(2)$ refers to the second FCT step. 
  This process can be repeated multiple times. 
  We refer the reader to \cite{kuzmin2012flux} and references therein for more details.
\end{rmk}

\section{Global monolithic convex (GMC) limiting} \label{sec:gmc}
In this section we use a different technique to determine limiters
$\alpha_{ij}$ in \eqref{lfs-cf} that will guarantee the MPP property \eqref{global_bounds}.
Namely, we follow
the global monolithic convex (GMC) limiting approach from \cite{exGMCL}.
As in the previous section, $G_{ij}^L$ and $G_{ij}^H$ are the low- and
high-order fluxes given by \eqref{low-order-fluxes} and
\eqref{high-order-fluxes}, respectively. 
Before defining the limiters $\alpha_{ij}$, we need to rewrite scheme \eqref{lfs-cf}
in a form like that given in \cite[Section 3.2]{exGMCL}. Let us define the following quantities:
\begin{align*}
  \lambdaS_i := \sum_{\jinN}|S_{ij}|\lambda_{ij}, 
  \qquad
  \bar u_i(u) := \frac{1}{\lambdaS_i}\sum_{\jinN}|S_{ij}|\lambda_{ij}\bar u_{ij}(u).
\end{align*}
In terms of these, the low-order scheme \eqref{low-order-scheme_bar_states} can be written as
\begin{align}\label{low_gmc}
  u_i^{L,n+1} = u_i^{L,n} + \frac{\Delta t}{|K_i|}\sum_{\jinN}|S_{ij}|\lambda_{ij}\left(\bar u_{ij}(u^{L,n+1})-u_i^{L,n+1}\right)
  = u_i^{L,n} + \frac{\Delta t}{|K_i|}\lambdaS_i\left(\bar u_i(u^{L,n+1})-u_i^{L,n+1}\right).
\end{align}

Let us consider \eqref{limited-flux-scheme} where the high-order flux $G_{ij}^H$ is 
computed via \eqref{high-order-fluxes} at the beginning of the time step. 
The low-order flux $G_{ij}^L$ is given by \eqref{low-order-fluxes} and is treated implicitly; i.e., 
$G_{ij}^L=G_{ij}^L(u^{n+1})$. Using \eqref{low_gmc}, we rewrite scheme 
\eqref{limited-flux-scheme} as follows:
\begin{align*}
  u_i^{n+1} 
  = u_i^n + \frac{\Delta t}{|K_i|}\lambdaS_i \left[\bar u_i^*(u^{n+1})-u_i^{n+1}\right],
\end{align*}
where
\begin{align}\label{ustar}
  \bar u_i^*(u^{n+1})
  = \bar u_i(u^{n+1})+ \frac{1}{\lambdaS_i} \sum_{\jinN} |S_{ij}|\alpha_{ij}(G_{ij}^L(u^{n+1})-G_{ij}^H).
\end{align}
We define 
\begin{align}\label{g_for_gmc}
  g_i(u^{n+1}) := u_i^{n+1}+\frac{\bar u_i^*(u^{n+1})-u_i^{n+1}}{1+\gamma},
\end{align}
where $\gamma\geq 0$ is a constant that can be adjusted to improve accuracy; see \cite{exGMCL}. 
Finally, scheme \eqref{limited-flux-scheme} becomes
\begin{align}\label{gmc}
  u_i^{n+1} 
  = u_i^n + \frac{\Delta t}{|K_i|}\lambdaS_i(1+\gamma) \left[g_i(u^{n+1})-u_i^{n+1}\right].
\end{align}
The MPP properties of \eqref{gmc} are guaranteed by the following theorem. 
\begin{thm}(Maximum principle)
  Let 
  \begin{align}\label{gmc_pmQ}
    Q_i^-:=\lambdaS_i \left[(u^{\min}-\bar u_i(u^{n+1}))+\gamma(u^{\min}-u_i^{n+1})\right], \qquad
    Q_i^+:=\lambdaS_i \left[(u^{\max}-\bar u_i(u^{n+1}))+\gamma(u^{\max}-u_i^{n+1})\right].
  \end{align}
  Assume $u_i^n\in[u^{\min},u^{\max}]$ and that $\alpha_{ij}$'s are chosen to satisfy
  \begin{align}\label{gmc_bounds}
    Q_i^- \leq \sum_{\jinN}|S_{ij}|\alpha_{ij}(G_{ij}^L-G_{ij}^H) \leq Q_i^+.
  \end{align}
  Then $u_i^{n+1}$ given by \eqref{gmc} satisfies $u_i^{n+1}\in[u^{\min}, u^{\max}]$ with no time step restriction.
\end{thm}

\begin{proof}
  Considering definition \eqref{ustar} and the bounds \eqref{gmc_bounds}, we get
  \begin{align*}
    \bar u_i^*-u_i^{n+1}\leq \bar u_i + \frac{Q_i^+}{\lambdaS_i} - u_i^{n+1}
    = (1+\gamma)(u^{\max}-u_i^{n+1}).
  \end{align*}
  Using this upper bound within definition \eqref{g_for_gmc}, we get $g_i(u^{n+1}) \leq u^{\max}$.
  Since $u_i^n, ~g_i(u^{n+1})\leq u^{\max}$ and using \eqref{gmc}, we get
  \begin{align*}
    u_i^{n+1} 
    = u_i^n + \frac{\Delta t}{|K_i|}\lambdaS_i(1+\gamma) \left[g_i(u^{n+1})-u_i^{n+1}\right]
    \leq 
    u_i^{\max} + \frac{\Delta t}{|K_i|}\lambdaS_i(1+\gamma) \left[u^{\max}-u_i^{n+1}\right]
    \implies u_i^{n+1} \leq u_i^{\max}.
\end{align*}
  The lower bound is proven similarly. 
\end{proof}

We must choose the limiters $\alpha_{ij}$ to satisfy \eqref{gmc_bounds}. 
We do that via algorithm \eqref{limiters} with $Q_i^{\pm}$ given by \eqref{gmc_pmQ}. 
To prove conservation of mass by \eqref{gmc}, consider
\begin{align*}
  \sum_i \lambdaS_i (1+\gamma)\left[g_i(u^{n+1})-u_i^{n+1}\right]
  &=
  \sum_i \lambdaS_i \left[\bar u_i^* -u_i^{n+1}\right]=
  \sum_i \lambdaS_i \left[\bar u_i -u_i^{n+1}\right]
  +\underbrace{\sum_i\sum_{\jinN}|S_{ij}|\alpha_{ij}(G_{ij}^L-G_{ij}^H)}_{\textstyle=0}
  \\
  &=\sum_i \sum_{\jinN}|S_{ij}|\lambda_{ij}\left(\bar u_{ij}-u_i^{n+1}\right)
  =-\sum_i\sum_{\jinN}|S_{ij}|G_{ij}^L=0.
\end{align*}
Therefore, $\sum_i |K_i|u_i^{n+1} = \sum_i |K_i| u_i^n \implies \sum_i |K_i|u_i^{n+1} = \sum_i |K_i| u_i^0$. 

Due to the highly nonlinear nature of \eqref{ustar}, using Newton's method 
to solve \eqref{gmc} with an exact Jacobian is difficult. 
Instead, we have found the following fixed point iteration to be useful:
\begin{align}\label{gmc_iter}
  u_i^{(k+1)} = u_i^n + \frac{\Delta t}{|K_i|}\lambdaS_i(1+\gamma)\left[g_i\left(u^{(k)}\right)-u_i^{(k+1)}\right]
  \implies
  u_i^{(k+1)} = \frac{1}{1+\frac{\Delta t}{|K_i|}\lambdaS_i(1+\gamma)}\left[u_i^n + \frac{\Delta t}{|K_i|}\lambdaS_i(1+\gamma) g_i\left(u^{(k)}\right)\right],
\end{align}
with $u^{(0)}=u^n$. For each time step, we run this iterative algorithm until  
\begin{align*}  
  \left|\left| 
  u_i^{(k+1)} - u_i^n - \frac{\Delta t}{|K_i|}\lambdaS_i(1+\gamma)\left[g_i\left(u^{(k+1)}\right)-u_i^{(k+1)}\right] 
  \right|\right|_{\ell^2} 
  \leq \text{tol}^{\GMC} = 10^{-12}.
\end{align*}


\section{Maximum principle preservation for intermediate stages}\label{sec:int_stages}
The procedures outlined in Sections \ref{sec:fct} and \ref{sec:gmc} guarantee preservation of the maximum principle
for the new solution $u^{n+1}$, but not necessarily for the intermediate stages
$y^m$.  For some applications, it may be important to guarantee the maximum principle
for the intermediate stages, particularly if the
system is not defined for values outside certain bounds.
We consider two possible approaches:

\begin{enumerate}
\item{\it Diagonally implicit Runge-Kutta methods}

Apply the limiting procedure to each stage value of a given DIRK scheme. 
This can (at least formally) reduce the order of accuracy of the time integration scheme. 
In \cite{exGMCL}, this approach was used with explicit Runge-Kutta methods. 
The authors did not observe loss of accuracy in their numerical experiments.

\item{\it Methods with SSP stages}

Start with a spatial semi-discretization that is MPP and
let $\Delta t^{\rm FE}$ denote the time step under which it is MPP when discretized by the forward Euler method.
Given a RK method, let $A$ denote the $M\times M$ matrix of the Butcher coefficients $a_{ms}$,
choose $\mu>0$, and set
$$
X(\mu)=\left(I+\mu A\right)^{-1}.
$$
Let $e$ denote the column vector of length $M$ with all entries equal to unity.
Then if
\begin{align} \label{SSP-stages}
  AX(\mu) \ge 0, \qquad AX(\mu)e \le e,
\end{align}
then the intermediate stages will be MPP
for $\Delta t\le \mu \Delta t^{\rm FE}$.
This approach is also considered in \cite{exGMCL}.
Note that the conditions \eqref{SSP-stages} are more relaxed than those required
for the full method to be SSP.
\end{enumerate}

\subsection{Implicit Euler extrapolation methods}
A particularly useful class of methods are those satisfying \eqref{SSP-stages}
for arbitrarily large values of $\mu$.  Such methods have stages that are unconditionally
SSP (i.e., SSP under any step size) and can be constructed using extrapolation
applied to the implicit Euler method \cite[Section IV.9]{Hairer:ODEs2}.  These methods are nearly A-stable
(specifically, they are $A(\alpha)$-stable with $\alpha$ close to 90 degrees)
and can be constructed to have any order of accuracy.  These methods are also
highly parallelizable \cite{2014_hork}.
The implicit Euler extrapolation algorithm for a method of order $p$ is given
in Algorithm \ref{alg:extrap}.  For any fixed $p$, this algorithm can be
written as a Runge-Kutta method.  
As an example, we provide coefficients of the 4th-order method below.
The coefficients are given in the standard Butcher form, although the
implementation is done more efficiently using Algorithm \ref{alg:extrap}.

\begin{align}\label{implicit_euler_method}
\begin{tabular}{c|cccccccccc}
  $1$ & $1$ &  &  &  &  &  &  &  &  & \\ [5pt]
  $\frac{1}{2}$ &  & $\frac{1}{2}$ &  &  &  &  &  &  &  & \\ [5pt]
  $1$ &  & $\frac{1}{2}$ & $\frac{1}{2}$ &  &  &  &  &  &  & \\ [5pt]
  $\frac{1}{3}$ &  &  &  & $\frac{1}{3}$ &  &  &  &  &  & \\ [5pt]
  $\frac{2}{3}$ &  &  &  & $\frac{1}{3}$ & $\frac{1}{3}$ &  &  &  &  & \\ [5pt]
  $1$ &  &  &  & $\frac{1}{3}$ & $\frac{1}{3}$ & $\frac{1}{3}$ &  &  &  & \\ [5pt]
  $\frac{1}{4}$ &  &  &  &  &  &  & $\frac{1}{4}$ &  &  & \\ [5pt]
  $\frac{1}{2}$ &  &  &  &  &  &  & $\frac{1}{4}$ & $\frac{1}{4}$ &  & \\ [5pt]
  $\frac{3}{4}$ &  &  &  &  &  &  & $\frac{1}{4}$ & $\frac{1}{4}$ & $\frac{1}{4}$ & \\ [5pt] 
  $1$ &  &  &  &  &  &  & $\frac{1}{4}$ & $\frac{1}{4}$ & $\frac{1}{4}$ & $\frac{1}{4}$ \\ [5pt] \hline \\ [-10pt]
  & $-\frac{1}{6}$ & $2$ & $2$ & $-\frac{9}{2}$ & $-\frac{9}{2}$ & $-\frac{9}{2}$ & $\frac{8}{3}$ & $\frac{8}{3}$ & $\frac{8}{3}$ & $\frac{8}{3}$ 
\end{tabular}
\end{align}

\begin{algorithm}\caption{Implicit Euler extrapolation of order $p$ ({\bf IEX$p$})}
  \label{alg:extrap}
  \begin{algorithmic}

\For{$k = 1 \to p$}  \Comment{Compute first order approximations}
    \State $Y_{k0} = u^n$
    \For{$j=1 \to k$}
        \State Solve $Y_{kj} = Y_{k,j-1} + \frac{h}{k}f(Y_{kj})$
    \EndFor
    \State $T_{k1} = Y_{kk}$
\EndFor

\For{$k=2 \to p$}  \Comment{Extrapolate to get higher order}
    \For{$j=k \to p$}
        \State $T_{jk} = T_{j,k-1} + \frac{T_{j,k-1}-T_{j-1,k-1}}{\frac{j}{j-k+1}-1}$
        \Comment{Aitken-Neville formula for extrapolation to order k}
    \EndFor
\EndFor
\State $u^{n+1} = T_{pp}$ \Comment{New solution value}
\end{algorithmic}
\end{algorithm}

For the second approach above, we need a spatial semi-discretization that is MPP, which can be
obtained by applying the GMC limiters from Section \ref{sec:gmc} only to the spatial discretization. 
We refer to \cite[Section 2.2]{exGMCL} for details. 
In Section \ref{sec:num_linear_convection_diffusion_ext_euler}, 
we test this methodology with a linear advection-diffusion problem in one-dimension. 
We recover the full accuracy of the underlying high-order scheme. 

\section{Numerical examples}\label{sec:num}
In this section, we present a series of one- and two-dimensional numerical experiments to demonstrate 
the properties of the MPP algorithms we propose. For each problem, we consider the following four methods:
\begin{itemize}
  \item LLF-BE. Low-order (local Lax-Friedrichs) spatial discretization and backward Euler time integration; 
    see Section \ref{sec:low_order} for details.
  \item WENO-SDIRK. Fifth-order WENO spatial discretization and a fifth-order SDIRK time integration, 
    whose Butcher tableau is given below; see Section \ref{sec:high_order} for details.
    This high-order scheme is used as the baseline high-order method for the following two MPP algorithms. 
  \item FCT-SDIRK. MPP algorithm presented in Section \ref{sec:fct}. 
  \item GMC-SDIRK. MPP algorithm presented in Section \ref{sec:gmc}. 
\end{itemize}
The fifth-order SDIRK method that we consider, which can be found in 
\cite[Section 7.2.2]{kennedy2016diagonally} and references therein, has the following Butcher tableau:
\begin{align}
\begin{tabular}{c|ccccc}
$\frac{4024571134387}{14474071345096}$ & $\frac{4024571134387}{14474071345096}$ & $0$ & $0$ & $0$ & $0$ \\ [5pt]
$\frac{5555633399575}{5431021154178}$  & $\frac{9365021263232}{12572342979331}$ & $\frac{4024571134387}{14474071345096}$ & $0$ & $0$ & $0$ \\ [5pt]
$\frac{5255299487392}{12852514622453}$ & $\frac{2144716224527}{9320917548702}$  & $\frac{-397905335951}{4008788611757}$  & $\frac{4024571134387}{14474071345096}$ & $0$ & $0$ \\ [5pt]
$\frac{3}{20}$ & $\frac{-291541413000}{6267936762551}$  & $\frac{226761949132}{4473940808273}$   & $\frac{-1282248297070}{9697416712681}$ & $\frac{4024571134387}{14474071345096}$ & $0$ \\ [5pt]
$\frac{10449500210709}{14474071345096}$ & $\frac{-2481679516057}{4626464057815}$ & $\frac{-197112422687}{6604378783090}$  & $\frac{3952887910906}{9713059315593}$  & $\frac{4906835613583}{8134926921134}$ & $\frac{4024571134387}{14474071345096}$ \\ [5pt] \hline \\ [-10pt]
& $\frac{-2522702558582}{12162329469185}$ & $\frac{1018267903655}{12907234417901}$ & $\frac{4542392826351}{13702606430957}$ & $\frac{5001116467727}{12224457745473}$ & $\frac{1509636094297}{3891594770934}$.
\end{tabular}
\end{align}
In addition to the previous MPP algorithms, for the first one-dimensional test that we solve, 
we consider an algorithm that guarantees the intermediate solutions of the Runge-Kutta scheme are 
MPP. We do that only for one test to demonstrate that preserving the maximum principle for the intermediate stages 
does not destroy the accuracy properties of the underlying high-order scheme. 

The spatial discretization is performed on uniform grids with $N_h$ elements.
Let $K_i$ denote the $i$-th element; then
$K_i=[x_{i-1/2},x_{i+1/2}]$ and
$K_i=\{(x,y)\in\mathbb{R} ~|~ x\in[x_{i-1/2},x_{i+1/2}], y\in[y_{i-1/2},y_{i+1/2}]\}$ 
for the one- and two-dimensional domains, respectively. The mesh spacing is denoted by 
$\Delta x$ and $\Delta y$ in the x- and y-direction, respectively. 
For the discretization in time, we use by default $\Delta t=0.5 \Delta x$. 
To quantify the magnitude of the overshoots and undershoots, we report 
\begin{align*}
  \delta
  =\min
  \left\{
  \min_{i,n} (u^n_i-u^{\min}), ~\min_{i,n} (u^{\max} - u^n_i)
  \right\}.
\end{align*}
Note that $\delta\geq 0$ for any MPP solution. In practice, however, $\delta$ might be a small 
negative number on the order of machine precision, which indicates a small violation of the maximum principle. 
In practice it might be acceptable to clip these values, 
since the methods are conservative only up to machine precision.
If the exact solution is available, we calculate and report the $L_1$ error
\begin{align*}
  E_1(t) = |K_i| \sum_{i=1}^{N_h}\left|\tilde u_i(t)-u^{\text{exact}}(x_i,y_i,t)\right|,
\end{align*}
where $\tilde u_i(t)$ is a fifth-order polynomial reconstruction of the numerical solution evaluated 
at $(x,y)=(x_i, y_i)$.
In addition, we report the corresponding experimental order of convergence (EOC). 

\subsection{Linear convection-diffusion}\label{sec:num_linear_convection_diffusion}
We start with the linear problem proposed in \cite{yang2016high}. 
The problem is given by 
\begin{subequations}\label{lin_conv_diff_prob}
\begin{align}
  u_t + a u_x &= \epsilon u_{xx}, \quad x\in[0,2\pi], \\
  u(x,0) &= \sin^4(2\pi x),
\end{align}
\end{subequations}
with periodic boundary conditions.
The coefficient $\epsilon$ controls the amount of dissipation and $a$ is the speed of advection.
We take $\epsilon=\{0,0.001\}$ and $a=1$. The exact solution, 
also found in \cite{yang2016high}, is 
\begin{align*}
  u(x,t)=\frac{3}{8}-\frac{1}{2}\exp(-4\epsilon t)\cos(2(x-t))+\frac{1}{8}\exp(-16\epsilon t)\cos(4(x-t)).
\end{align*}
We solve the problem up to the final time $t=2\pi$ using $\lambda_{i+1/2}^A=1$ for all $i$. 
The global bounds are given by $u^{\min}=0$ and $u^{\max}=1$.
The results of a convergence study are summarized in Tables \ref{table:linear1} and \ref{table:linear2}. 
Note that the high-order WENO-SDIRK method produces undershoots and/or overshoots in both cases, 
which is indicated by the negative values of $\delta$. The rest of the methods (Low-BE, FCT-SDIRK and GMC-SDIRK)
produce MPP solutions. To achieve full accuracy when $\epsilon=0$, 
we require at least 2 iterations with the FCT limiters and $\gamma\geq 2$ with the GMC limiters. 
In contrast, when $\epsilon=0.001$, 
the physical dissipation reduces the action of the limiters, which leads to full accuracy 
with only one iteration when the FCT limiters are used and $\gamma=0$ when the GMC limiters are used. 

\begin{table}[!h]\scriptsize
  \begin{center}
    \subfloat[Low-order in space via BE]{
      \begin{tabular}{||c|c|c|c||} \hline
        $\Delta x$ & $E_1$ & rate & $\delta$  \\ \hline
        1/25  & 2.04 &  --  & 7.17E-03 \\
        1/50  & 1.85 & 0.14 & 6.89E-04 \\
        1/100 & 1.42 & 0.38 & 4.92E-05 \\
        1/200 & 9.42E-01 & 0.59 & 3.19E-06 \\ \hline
      \end{tabular}
    } \qquad   
    \subfloat[WENO-SDIRK]{
      \begin{tabular}{||c|c|c|c||} \hline
        $\Delta x$ & $E_1$ & rate & $\delta$ \\ \hline
        1/25  & 2.73E-01 &  --  & -2.30E-02 \\ 
        1/50  & 1.98E-02 & 3.79 & -2.18E-03 \\
        1/100 & 2.20E-03 & 3.16 & -2.42E-04 \\
        1/200 & 1.25E-04 & 4.15 & -2.08E-05 \\ \hline
      \end{tabular} 
    } 

    \subfloat[FCT-SDIRK with different number of iterations]{
      \begin{tabular}{||c||c|c|c||c|c|c||} \hline
        & \multicolumn{3}{c||}{With 1 iter}
        & \multicolumn{3}{c||}{With 2 iter} \\ \hline
        $\Delta x$ & $E_1$ & rate & $\delta$ & $E_1$ & rate & $\delta$ \\ \hline
        1/25  & 2.45E-01 &  --  & 1.26E-03 & 2.39E-01 &  --  & -1.73e-18 \\
        1/50  & 2.07E-02 & 3.56 & 1.66E-04 & 1.96E-02 & 3.60 & -4.34e-19 \\
        1/100 & 2.09E-03 & 3.31 & 1.27E-05 & 2.04e-03 & 3.27 & -2.71e-20 \\
        1/200 & 1.66E-04 & 3.65 & 8.35E-07 & 1.15e-04 & 4.15 & -1.69e-21 \\ \hline
      \end{tabular}
    }
    
    \subfloat[GMC-SDIRK with different values of $\gamma$]{
      \begin{tabular}{||c||c|c|c||c|c|c||c|c|c||} \hline
        & \multicolumn{3}{c||}{$\gamma=0$}
        & \multicolumn{3}{c||}{$\gamma=1$}
        & \multicolumn{3}{c||}{$\gamma=2$} \\ \hline
        $\Delta x$ & $E_1$ & rate & $\delta$ & $E_1$ & rate & $\delta$ & $E_1$ & rate & $\delta$ \\ \hline
        1/25  & 2.58E-01 &   -- &  4.30E-04 & 2.42E-01 &  --  & 2.98E-04 & 2.41E-01 &  --  & 2.28E-04 \\
        1/50  & 2.74E-02 & 3.24 &  2.75E-05 & 2.07E-02 & 3.55 & 1.90E-05 & 1.99E-02 & 3.59 & 1.46E-05 \\
        1/100 & 3.46E-03 & 2.98 &  1.73E-06 & 2.04E-03 & 3.34 & 1.20E-06 & 2.06E-03 & 3.27 & 9.15E-07 \\
        1/200 & 4.03E-04 & 3.10 &  1.08E-07 & 1.54E-04 & 3.73 & 7.49E-08 & 1.16E-04 & 4.15 & 5.73E-08 \\ \hline
      \end{tabular}
    }
    \caption{Grid convergence study for the linear problem \eqref{lin_conv_diff_prob} with $\epsilon=0$.\label{table:linear1}}
  \end{center}
\end{table}

\begin{table}[!h]\scriptsize
  \begin{center}
    \begin{tabular}{||c||c|c|c||c|c|c||c|c|c||c|c|c||} \hline
      & \multicolumn{3}{c||}{Low-BE} 
      & \multicolumn{3}{c||}{WENO-SDIRK} 
      & \multicolumn{3}{c||}{FCT-SDIRK (1 iter)} 
      & \multicolumn{3}{c||}{GMC-SDIRK ($\gamma=0$)} \\ \hline 
      $\Delta x$ & $E_1$ & rate & $\delta$ & $E_1$ & rate & $\delta$ & $E_1$ & rate & $\delta$  & $E_1$ & rate & $\delta$  \\ \hline
      1/25  & 1.98     &  --  & 7.23e-03 & 2.49E-01 &  --  & -1.88E-02 & 2.25E-01 &  --  & 1.28E-03 & 2.36E-01 &  --  & 4.29E-04 \\
      1/50  & 1.80     & 0.14 & 7.01e-04 & 1.58E-02 & 3.98 & -9.01E-04 & 1.67E-02 & 3.76 & 1.71E-04 & 1.87E-02 & 3.66 & 2.73E-05 \\
      1/100 & 1.37     & 0.38 & 5.09e-05 & 1.25E-03 & 3.66 & -3.86E-05 & 1.27E-03 & 3.71 & 1.34E-05 & 1.28E-03 & 3.86 & 1.70E-06 \\
      1/200 & 9.07E-01 & 0.59 & 3.40e-06 & 5.46E-05 & 4.52 & -1.14E-06 & 5.48E-05 & 4.54 & 9.35E-07 & 5.48E-05 & 4.55 & 1.05E-07 \\ \hline
    \end{tabular}
    \caption{Grid convergence study for the linear problem \eqref{lin_conv_diff_prob} with $\epsilon=0.001$.\label{table:linear2}}
  \end{center}
\end{table}


We also conducted experiments for the pure diffusion problem, taking $a=0$.
In this case, the high-order discretization is MPP, so the limiters are not needed
(and do not turn on).

\subsubsection{Linear convection-diffusion via an implicit Euler extrapolation method}
\label{sec:num_linear_convection_diffusion_ext_euler}
Here we again solve the linear problem \eqref{lin_conv_diff_prob} with $\epsilon=0.001$,
using WENO reconstruction with GMC limiters applied only to the semi-discretization. 
The high-order time integration is given by a 4th-order 
implicit Euler extrapolation method (with Butcher tableau \eqref{implicit_euler_method}). 
Since the intermediate stages are unconditionally strong stability preserving, 
each intermediate solution is MPP. To guarantee the RK update is also MPP, 
we employ the methodology from Section \ref{sec:gmc}. 
The results of a convergence study are summarized in Table \ref{table:linear2_with_implicit_euler}. 

\begin{table}[!h]\scriptsize
  \begin{center}
    \begin{tabular}{||c||c|c|c||} \hline
      $\Delta x$ & $E_1$ & rate & $\delta$ \\ \hline
      1/25  & 2.59e-01 &  --  & 2.37e-04 \\
      1/50  & 2.09e-02 & 3.63 & 2.73e-05 \\
      1/100 & 1.41e-03 & 3.89 & 1.70e-06 \\
      1/200 & 5.81e-05 & 4.61 & 1.05e-07 \\ \hline
    \end{tabular}
    \caption{Grid convergence study for the linear problem \eqref{lin_conv_diff_prob} with $\epsilon=0.001$.
      In this case we use a method that guarantees the intermediate solutions of the RK scheme are MPP. 
      See the third approach in Section \ref{sec:int_stages} for details. 
      \label{table:linear2_with_implicit_euler}}
  \end{center}
\end{table}

\subsection{Viscous Burgers' equation}\label{sec:num_burgers}
Let us consider now the viscous Burgers' equation 
\begin{subequations}\label{burgers}
\begin{align}
  u_t+\left(\frac{u^2}{2}\right) = \epsilon u_{xx}, \quad x\in[-1,1],
\end{align}
with $\epsilon=0.01$ and periodic boundary conditions. Similarly to \cite{yang2016high}, we use the following initial condition 
\begin{align}
  u(x,0) = 
  \begin{cases}
    2, & \mbox{ if } |x|<0.5, \\
    0, & \mbox{ otherwise}. 
  \end{cases}
\end{align}
\end{subequations}
For this problem, we use $\lambda_{i+1/2}^A=\max\{u_i,u_{i+1},\hat u_i^+, \hat u_{i+1}^-\}$. 
The global bounds are given by $u^{\min}=0$ and $u^{\max}=2$.
In Figure \ref{fig:burgers}, we show the results at $t=0.25$ using the different methods and two refinements. 
The baseline high-order WENO scheme produces undershoots and/or overshoots, which are eliminated 
(up to machine precision) by all of the MPP algorithms. 

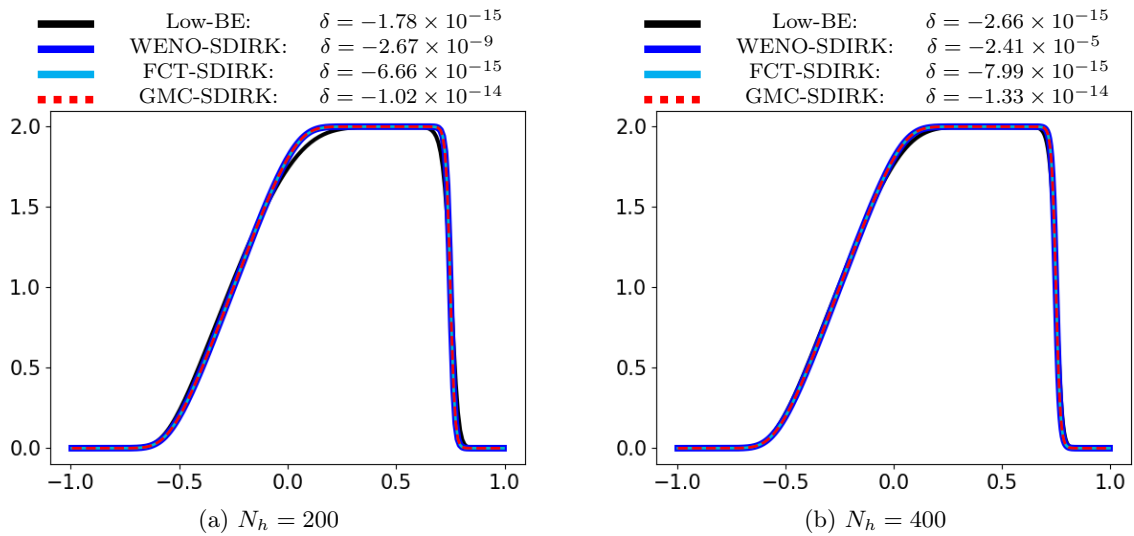
\begin{figure}[!ht]
  \begin{center}
  {\scriptsize
    \subfloat[$N_h=200$]{    
      \begin{tabular}{c}
        \begin{tabular}{ccc}
          \begin{tikzpicture}
            \draw[line width=1mm, black ,-] (0,0.) to (0.75,0.);
          \end{tikzpicture}
          & Low-BE:     & $\delta=-1.78 \times 10^{-15}$ \\
          \begin{tikzpicture}
            \draw[line width=1mm, blue ,-] (0,0.) to (0.75,0.);
          \end{tikzpicture}
          & WENO-SDIRK: & $\delta=-2.67 \times 10^{-9}~$ \\
          \begin{tikzpicture}
            \draw[line width=1mm, cyan ,-] (0,0.) to (0.75,0.);
          \end{tikzpicture}
          & FCT-SDIRK:   & $\delta=-6.66 \times 10^{-15}$ \\
          \begin{tikzpicture}
            \draw[line width=1mm, red, dashed] (0,0.) to (0.75,0.);
          \end{tikzpicture}
          & GMC-SDIRK:   & $\delta=-1.02 \times 10^{-14}$ 
        \end{tabular} \\
        \includegraphics[scale=0.5]{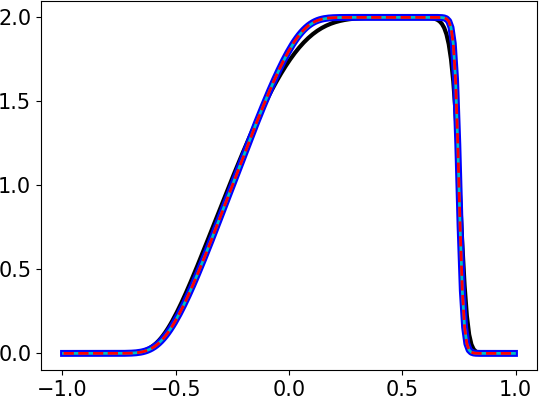}
      \end{tabular}
    }
    \qquad
    \subfloat[$N_h=400$]{
      \begin{tabular}{c}
        \begin{tabular}{ccc}
          \begin{tikzpicture}
            \draw[line width=1mm, black ,-] (0,0.) to (0.75,0.);
          \end{tikzpicture}
          & Low-BE:     & $\delta=-2.66 \times 10^{-15}$ \\
          \begin{tikzpicture}
            \draw[line width=1mm, blue ,-] (0,0.) to (0.75,0.);
          \end{tikzpicture}
          & WENO-SDIRK: & $\delta=-2.41 \times 10^{-5}~$ \\
          \begin{tikzpicture}
            \draw[line width=1mm, cyan ,-] (0,0.) to (0.75,0.);
          \end{tikzpicture}
          & FCT-SDIRK:   & $\delta=-7.99 \times 10^{-15}$ \\
          \begin{tikzpicture}
            \draw[line width=1mm, red, dashed] (0,0.) to (0.75,0.);
          \end{tikzpicture}
          & GMC-SDIRK:   & $\delta=-1.33 \times 10^{-14}$ 
        \end{tabular} \\
        \includegraphics[scale=0.5]{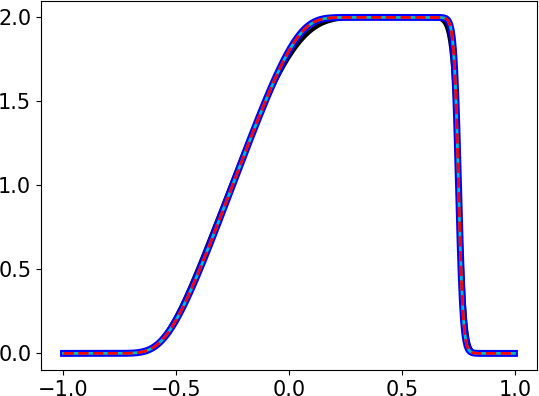}
      \end{tabular}
    }
  }
  \end{center}
  \caption{Numerical solution of the nonlinear problem \eqref{burgers}.
    Computations are
    performed using different number of cells.
    \label{fig:burgers}}
\end{figure}


\subsection{One-dimensional viscous Buckley-Leverett equation}\label{sec:num_BL}
Following \cite[Example 4.3]{yang2016high}, we consider the nonlinear problem 
\begin{subequations}\label{BL}
\begin{align}
  u_t+\bff(u)_x=\epsilon\left(c(u)u_x\right)_x, \quad x\in[0,1],
\end{align}
where $\epsilon=0.01$ and
\begin{align}
  \bff(u)=\frac{u^2}{u^2+(1-u)^2}, \qquad 
  c(u)=
  \begin{cases}
    4u(1-u), & \mbox{ if } 0\leq u\leq 1, \\
    0,       & \mbox{ otherwise }.
  \end{cases}
\end{align}
The boundary conditions are $u(0,t)=1$ and $u(1,t)=0$ and 
the initial condition is 
\begin{align}
  u(x,0) = 
  \begin{cases}
    1-3x, & \mbox{ if } 0\leq x<1/3, \\
    0, & \mbox{ otherwise}. 
  \end{cases}
\end{align}
\end{subequations}

As upper bound for the wave speed we use $\lambda_{i+1/2}^A=2$. The global bounds are given by
$u^{\min}=0$ and $u^{\max}=1$. In Figure \ref{fig:bl}, we show the solution at $t=0.2$ using 
the different methods and two refinements.
Using WENO-SDIRK, we get small undershoots and/or overshoots. 
The rest of the methods produce MPP solutions.

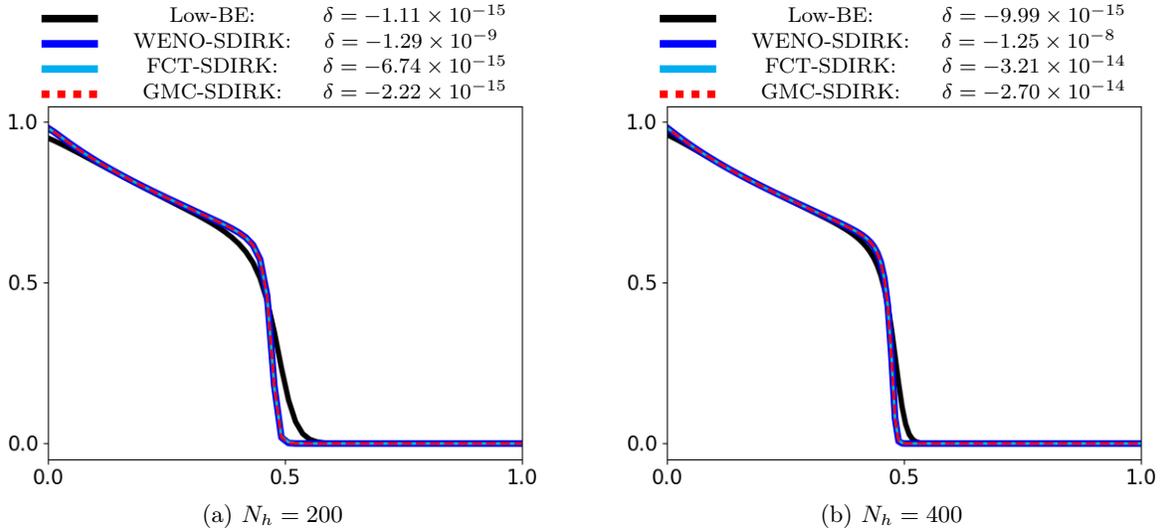
\begin{figure}[!ht]
  \begin{center}
  {\scriptsize
    \subfloat[$N_h=200$]{    
      \begin{tabular}{c}
        \begin{tabular}{ccc}
          \begin{tikzpicture}
            \draw[line width=1mm, black ,-] (0,0.) to (0.75,0.);
          \end{tikzpicture}
          &Low-BE:     & $\delta=-1.11 \times 10^{-15}$ \\
          \begin{tikzpicture}
            \draw[line width=1mm, blue ,-] (0,0.) to (0.75,0.);
          \end{tikzpicture}
          &WENO-SDIRK: & $\delta=-1.29 \times 10^{-9}~$ \\
          \begin{tikzpicture}
            \draw[line width=1mm, cyan ,-] (0,0.) to (0.75,0.);
          \end{tikzpicture}
          &FCT-SDIRK:   & $\delta=-6.74 \times 10^{-15}$ \\
          \begin{tikzpicture}
            \draw[line width=1mm, red, dashed] (0,0.) to (0.75,0.);
          \end{tikzpicture}
          &GMC-SDIRK:   & $\delta=-2.22 \times 10^{-15}$
        \end{tabular}\\
        \includegraphics[scale=0.5]{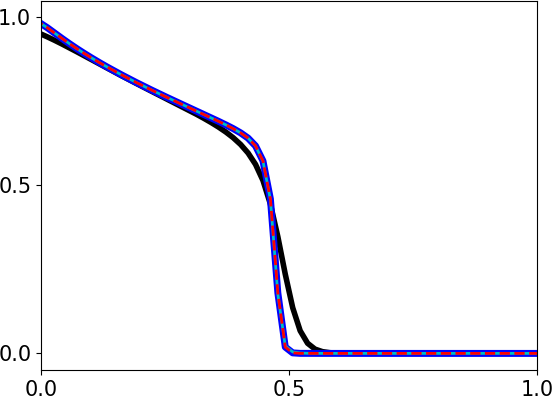}
      \end{tabular}
    }
    \qquad
    \subfloat[$N_h=400$]{
      \begin{tabular}{c}
        \begin{tabular}{ccc}
          \begin{tikzpicture}
            \draw[line width=1mm, black ,-] (0,0.) to (0.75,0.);
          \end{tikzpicture}
          & Low-BE:     & $\delta = -9.99 \times 10^{-15}$ \\
          \begin{tikzpicture}
            \draw[line width=1mm, blue ,-] (0,0.) to (0.75,0.);
          \end{tikzpicture}
          & WENO-SDIRK: & $\delta = -1.25 \times 10^{-8{\color{white}x}}$ \\
          \begin{tikzpicture}
            \draw[line width=1mm, cyan ,-] (0,0.) to (0.75,0.);
          \end{tikzpicture}
          & FCT-SDIRK:   & $\delta = -3.21 \times 10^{-14}$ \\
          \begin{tikzpicture}
            \draw[line width=1mm, red, dashed] (0,0.) to (0.75,0.);
          \end{tikzpicture}
          & GMC-SDIRK:   & $\delta = -2.70 \times 10^{-14}$ 
        \end{tabular}\\
        \includegraphics[scale=0.5]{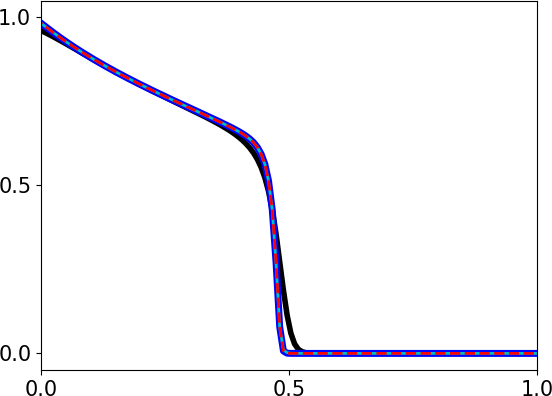}
      \end{tabular}
    }
  }
  \end{center}
  \caption{Numerical solution of the nonlinear problem \eqref{BL}.
    Computations are
    performed using different number of cells.
    \label{fig:bl}}
\end{figure}


\subsection{A one-dimensional steady state problem}\label{sec:steady_state}
Finally, we consider a problem with a steady state solution. Namely, we solve 
\begin{subequations}\label{steady_state_prob}
\begin{align}
  u_t+\bff(u,x)_x &= \epsilon u_{xx}, \quad x\in\mathbb{R}, \\
  \lim_{x\rightarrow \pm\infty} u(x) &= 0, 
\end{align}
\end{subequations}
with 
$\bff(u,x) = - \epsilon xu/\sigma^2$,
$\epsilon=0.01$ and $\sigma^2=0.01$.
It is easy to verify that 
\begin{align}\label{steady_state}
  u(x) = A\exp\left(-\frac{x^2}{2\sigma^2}\right) 
\end{align}
is the steady state solution of \eqref{steady_state_prob} where
the constant $A$ is determined by conservation of mass.
We take the computational domain to be $-1 \le x\le 1$ and invoke homogeneous
Dirichlet boundary conditions since $u(\pm 1)\approx 0$. 
As initial condition, we use
\begin{align*}
  u(x,0)=\sqrt{2\pi}\sigma  \sin^2(2\pi x),
\end{align*}
which leads to the steady state \eqref{steady_state} with amplitude $A=1$. 

For the flux function in this problem and with the initial condition that we consider, 
\eqref{steady_state_prob} satisfies a minimum principle.
Therefore, the MPP algorithms must guarantee $u\geq 0$. 
To guarantee positivity, we need the face states $\bar u_{ij}$, 
given by \eqref{dbar_state}, to be positive provided $u_i, u_j\geq 0$. 
From \eqref{dbar_state}, $c_{ij}=\epsilon\geq 0, ~\lambda_{ij}^A\geq 0 \implies \bar u_{ij}\geq 0$ 
provided $\uA_{ij}, ~\uD_{ij} \geq 0$. 
From \eqref{bar_state_D}, $\uD_{ij}$ is clearly non-negative if $u_i, u_j\geq 0$. 
We now find a condition on $\lambda_{ij}^A$ to guarantee $\uA_{ij}\geq 0$. 
On a one-dimensional grid, neighboring cells have $j=i+1$ or $j=i-1$,
and it is convenient to write $\uA_{i,i+1}=\uA_{i+1/2}$.
Let $v(x)=-\epsilon x/\sigma^2$. We get
\begin{align*}
  \uA_{i+1/2}
  &=\frac{u_i+u_{i+1}}{2}-\frac{\bff(u_{i+1},x_{i+1})-\bff(u_i,x_i)}{2\lambda_{i+1/2}^A} 
  =\frac{u_i+u_{i+1}}{2}-\frac{v_{i+1} u_{i+1}-v_iu_i}{2\lambda_{i+1/2}^A} \\
  &=\frac{1}{2\lambda_{i+1/2}^A}\left[(\lambda_{i+1/2}^A-v_{i+1})u_{i+1}+(\lambda_{i+1/2}^A+v_i)u_i\right].
\end{align*}
By choosing $\lambda_{i+1/2}^A\geq \max\{|v_i|,|v_{i+1}|\}$, 
we get $u_{i+1/2}^A\geq 0$ provided $u_i, ~u_{i+1}\geq 0$.
For simplicity, we use $\lambda_{i+1/2}^A=\epsilon/\sigma^2=1$. 
With respect to the global bounds, we use $u^{\min}=0$.
In Figure \ref{fig:steady}, we show the solution at different times using the different algorithms. 
In addition, we obtain the numerical solution at $t=20$ and perform a convergence test 
using \eqref{steady_state} as reference solution. The results are summarized in Table \ref{table:steady}.
For the coarser grids, the WENO-SDIRK method leads to small undershoots. 
The violations of the global bounds are eliminated by each of the MPP methods. 

\begin{figure}[!ht]
  \begin{center}
    \includegraphics[scale=0.275]{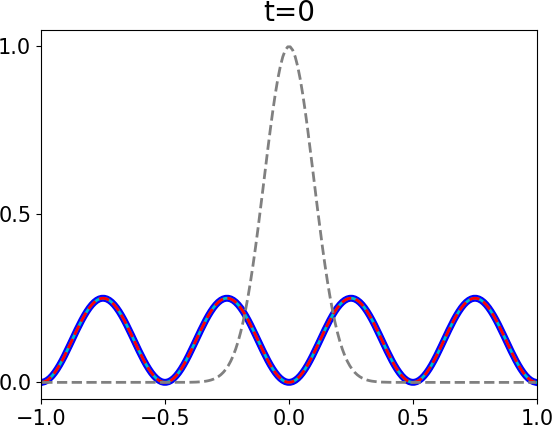}\qquad
    \includegraphics[scale=0.275]{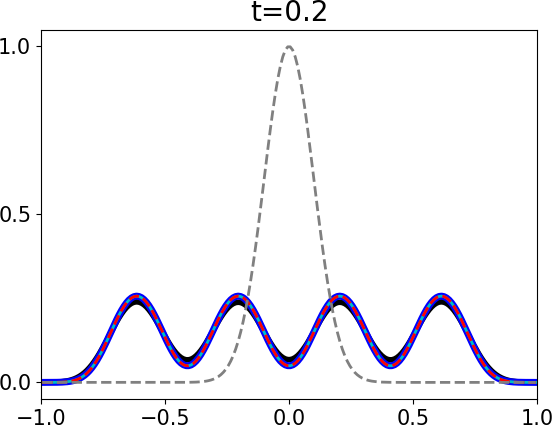}\qquad
    \includegraphics[scale=0.275]{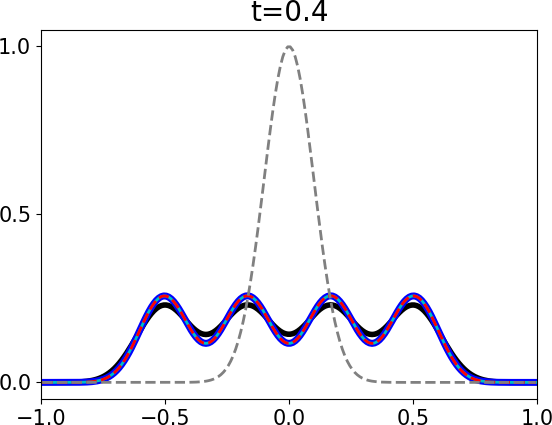}\qquad
    \includegraphics[scale=0.275]{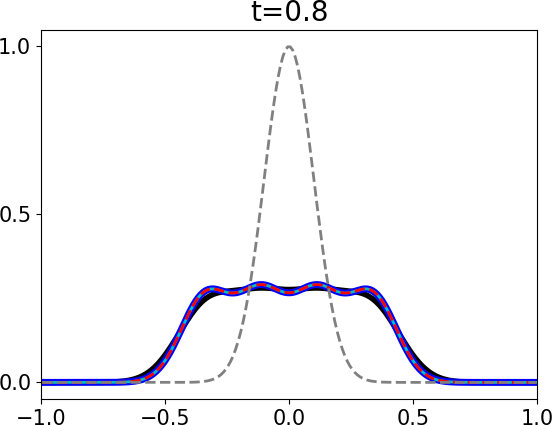} 

    \vspace{10pt}
    \includegraphics[scale=0.275]{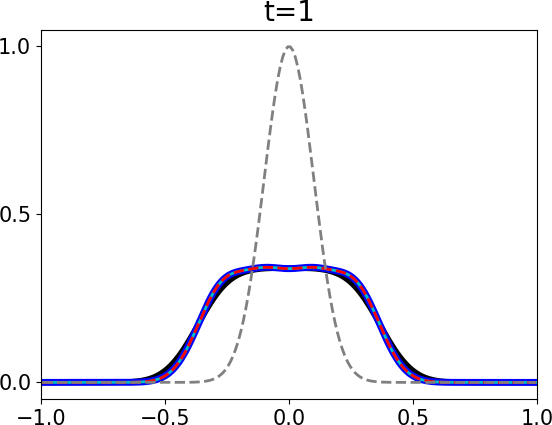}\qquad
    \includegraphics[scale=0.275]{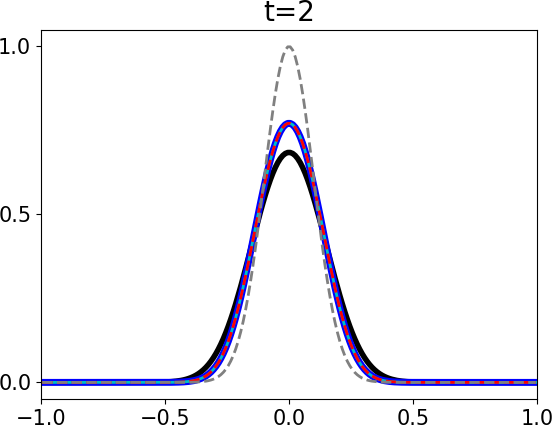}\qquad
    \includegraphics[scale=0.275]{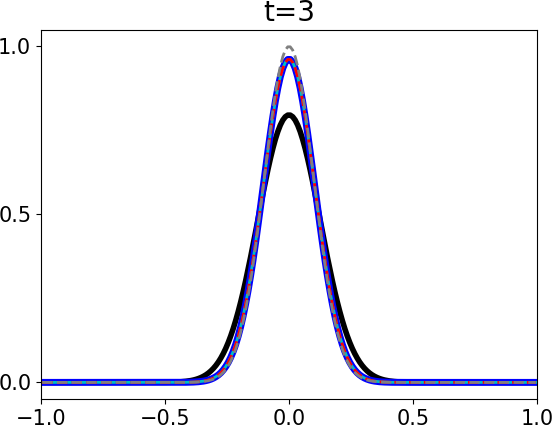}\qquad
    \includegraphics[scale=0.275]{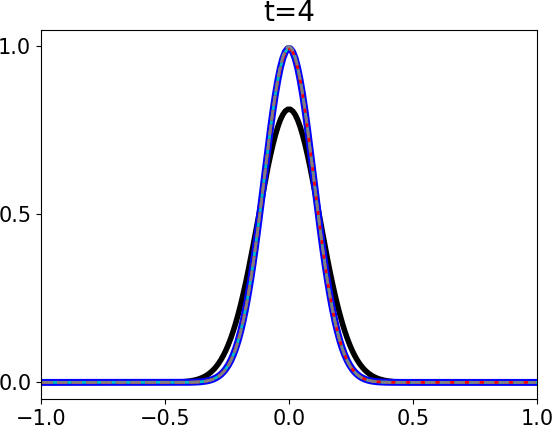} 
  \end{center}
  \caption{Numerical solution of the nonlinear problem \eqref{steady_state_prob} at different times.
    The simulation is done using $N_h=200$ cells. 
    The results with Low-BE, WENO-SDIRK, FCT-SDIRK and GMC-SDIRK are
    shown in black, blue, cyan and dashed red, respectively. 
    The dashed gray line is the steady state solution \eqref{steady_state}.
    \label{fig:steady}}
\end{figure}

\begin{table}[!h]\scriptsize
  \begin{center}
    \begin{tabular}{||c||c|c|c||c|c|c||c|c|c||c|c|c||} \hline
      & \multicolumn{3}{c||}{Low-BE} 
      & \multicolumn{3}{c||}{WENO-SDIRK} 
      & \multicolumn{3}{c||}{FCT-SDIRK} 
      & \multicolumn{3}{c||}{GMC-SDIRK} \\ \hline 
      $\Delta x$ & $E_1$ & rate & $\delta$ & $E_1$ & rate & $\delta$ & $E_1$ & rate & $\delta$  & $E_1$ & rate & $\delta$  \\ \hline
      1/25  & 1.87e-01 &   -- & 1.98E-05 & 3.24E-02 &   -- & -3.64e-03 & 3.24E-02 &   -- & -4.34e-19 & 3.24E-02 &  --  & -8.40e-26 \\
      1/50  & 1.31e-01 & 0.51 & 2.10E-08 & 2.09E-03 & 3.95 & -8.03e-05 & 2.09E-03 & 3.95 & -5.42e-20 & 2.09E-03 & 3.95 & -3.38e-30 \\
      1/100 & 8.34e-02 & 0.64 & 6.09E-12 & 1.27E-04 & 4.05 &  3.41e-23 & 1.27E-04 & 4.05 &  3.41e-23 & 1.27E-04 & 4.05 &  3.41e-23 \\
      1/200 & 4.91e-02 & 0.76 & 2.14E-15 & 2.64E-05 & 2.26 &  1.68e-22 & 2.64E-05 & 2.26 &  1.68e-22 & 2.64E-05 & 2.26 &  1.68e-22  \\ \hline
    \end{tabular}
    \caption{Grid convergence study for the nonlinear problem \eqref{steady_state_prob}.\label{table:steady}}
  \end{center}
\end{table}

\subsection{Two-dimensional solid rotation}\label{sec:num_2d_rotation}
The first two-dimensional test that we consider is the solid body rotation benchmark 
\cite{leveque1996high}, which is given by the variable-coefficient advection equation
\begin{align}\label{2D_solid_rot}
  u_t+[2\pi(0.5-y)u]_x+[2\pi(x-0.5)u]_y=0, \quad (x,y)\in[0,1]^2,
\end{align}
with periodic boundary conditions. The initial condition, shown in Figure \ref{fig:threeBody}, is 
\begin{subequations}\label{3_body_problem}
\begin{align}
  u(x,y,t=0) = 
  \begin{cases}
    u^{\text{hump}}(x,y) & \mbox{ if } ~\sqrt{(x-0.25)^2+(y-0.5)^2} \leq 0.15, \\
    u^{\text{cone}}(x,y) & \mbox{ if } ~\sqrt{(x-0.5)^2+(y-0.25)^2} \leq 0.15, \\
    1 & \mbox{ if } ~(x,y) \in \Omega^{\text{disk}}, \\
    0 & \mbox{ otherwise},
  \end{cases}
\end{align}
where 
\begin{align}
  u^{\text{hump}} = \frac{1}{4} + \frac{1}{4}\cos\left(\frac{\pi\sqrt{(x-0.25)^2+(y-0.5)^2}}{0.15}\right),
  \qquad
  u^{\text{cone}} = 1 - \frac{\sqrt{(x-0.5)^2+(y-0.25)^2}}{0.15}, \\
  \Omega^{\text{disk}} = \{(x,y)\in\mathbb{R} ~|~ \sqrt{(x-0.5)^2+(y-0.75)^2}\leq 0.15\}
  \backslash\{(x,y)\in\mathbb{R} ~|~ |x-0.5|< 0.025, y< 0.85\}.
\end{align}
\end{subequations}

\begin{figure}[!h]
  \begin{center}\scriptsize
    \includegraphics[scale=0.4]{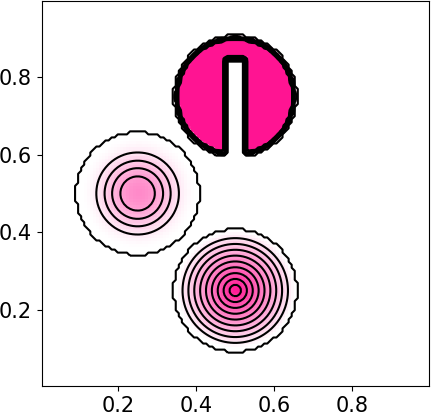}
  \end{center}
  \caption{Initial condition, given by \eqref{3_body_problem}, 
    for the solid rotation problem \eqref{2D_solid_rot}
    and the periodic vortex problem \eqref{2D_swirling}.
    \label{fig:threeBody}}
\end{figure}

The velocity field rotates the initial data around $(x,y)=(0.5, 0.5)$. 
After every revolution, the exact solution coincides with the initial condition. 
We solve the problem up to $t=1$. To facilitate the comparison against other high-order 
methods, like the ones in \cite{kuzmin2020entropy} and references therein, 
we use $N_h=128^2$ uniform square cells.
For simplicity, we use $\lambda_{ij}^A=\pi$. The global bounds are $u^{\min}=0$ and $u^{\max}=1$.
The solution for the different methods is shown in Figure \ref{fig:2D_solid_rot}. 
As expected, WENO-SDIRK produces small undershoots and overshoots. 
Both FCT-SDIRK and GMC-SDIRK produce MPP solutions and preserve
similar accuracy than WENO-SDIRK.

\begin{figure}[!h]
  \begin{center}\scriptsize
    \subfloat[Low-BE]{
      \begin{tabular}{c}
        $E_1=1.32\times 10^{-1}$\\
        $\delta=9.70\times 10^{-25}$ \\
        $u\in[1.67\times 10^{-2}, 0.2772]$ \\
        \includegraphics[scale=0.35]{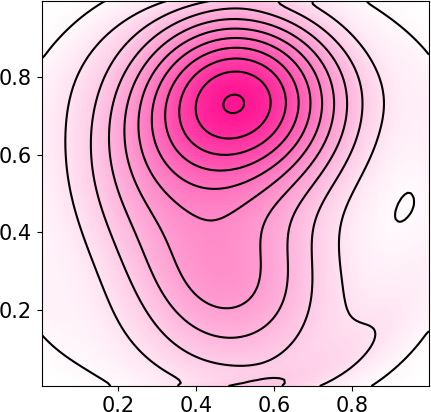}
      \end{tabular}} ~
    \subfloat[WENO-SDIRK]{
      \begin{tabular}{c}
        $E_1=2.09\times 10^{-2}$ \\
        $\delta= -4.05\times 10^{-2}$ \\
        $u\in[-7.80\times 10^{-4}, 1.0023]$ \\
        \includegraphics[scale=0.35]{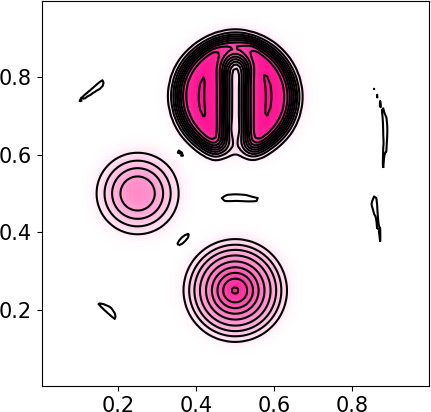}
    \end{tabular}} ~
    \subfloat[FCT-SDIRK]{
      \begin{tabular}{c}
        $E_1=2.22\times 10^{-2}$ \\
        $\delta=6.36\times 10^{-25}$ \\
        $u\in [7.85\times 10^{-12}, 0.9936]$ \\
        \includegraphics[scale=0.35]{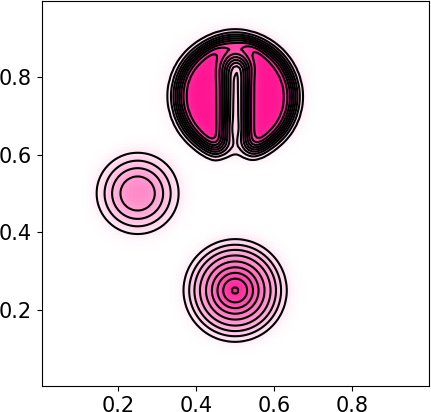}
    \end{tabular}} ~
    \subfloat[GMC-SDIRK]{
      \begin{tabular}{c}
        $E_1=2.10\times 10^{-2}$ \\
        $\delta=-1.81\times 10^{-22}$ \\
        $u\in[2.77\times 10^{-17}, 0.9998]$ \\
        \includegraphics[scale=0.35]{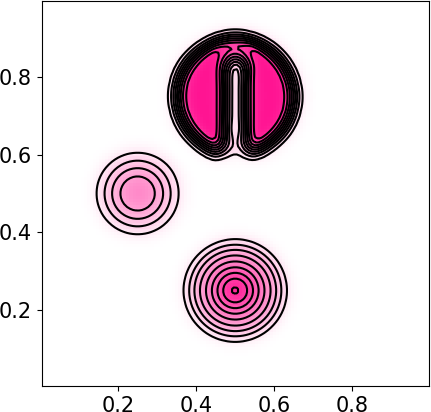}
    \end{tabular}}
  \end{center}
  \caption{Numerical solution (at $t=1$) of the linear problem \eqref{2D_solid_rot}
    with initial condition \eqref{3_body_problem} using the different schemes. 
    In all cases, we use $N_h=128^2$ uniform square cells. 
    The color scale in the plots goes from white to pink, which
    corresponds to $0$ and $1$, respectively. 
    \label{fig:2D_solid_rot}}
\end{figure}

\subsection{Two-dimensional periodic vortex}\label{sec:num_2D_periodic_vortex}
Let us now solve another benchmark problem proposed in \cite{leveque1996high}. 
The problem is given by 
\begin{align}\label{2D_swirling}
  u_t
  +[ \sin^2(\pi x)\sin(2\pi y)\cos(\pi t/T) u]_x
  -[ \sin^2(\pi y)\sin(2\pi x)\cos(\pi t/T) u]_y=0, \quad (x,y)\in[0,1]^2,
\end{align}
with periodic boundary conditions. The initial condition is the same as before; i.e., 
$u(x,y,t=0)$ is given by \eqref{3_body_problem}. 
From $t=0$ to $t=T/2$, the velocity field performs a swirling deformation to the initial condition. 
At $t=T/2$, the velocity reverses direction making the exact solution coincide with 
the initial condition at $t=T$. We solve the problem up to $t=T=1.5$ using $N_h=128^2$ uniform
square cells. For simplicity, we use $\lambda_{ij}^A=1$. The global bounds are $u^{\min}=0$ and $u^{\max}=1$. 
In Figure \ref{fig:2D_swirling}, we show the solution at $t=T/2$ and $t=T$ using the different methods. 
WENO-SDIRK violates the bounds while the rest of the schemes do not. The high-order accuracy of 
WENO-SDIRK is preserved by FCT-SDIRK and GMC-SDIRK.

\begin{figure}[!h]
  \begin{center}\scriptsize
    \subfloat[Low-BE]{
      \begin{tabular}{c}
        \includegraphics[scale=0.35]{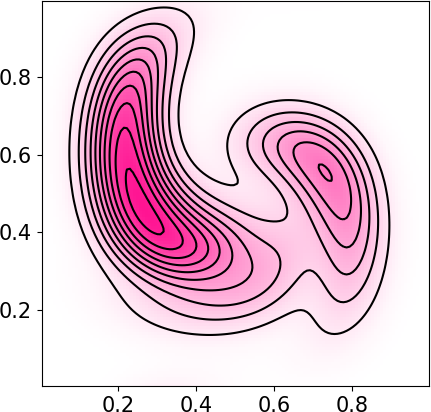} \\
        \vspace{3.5pt} \\ 
        $E_1=1.20\times 10^{-1}$\\
        $\delta=1.43\times 10^{-44}$ \\
        $u\in[8.52\times 10^{-4}, 0.3798]$ \\
        \includegraphics[scale=0.35]{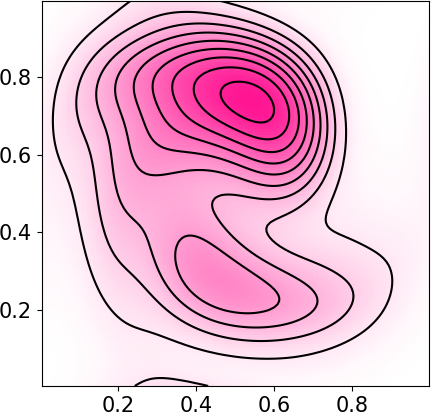}
      \end{tabular}} ~
    \subfloat[WENO-SDIRK]{
      \begin{tabular}{c}
        \includegraphics[scale=0.35]{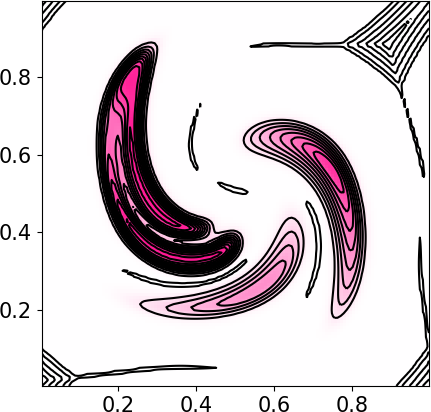} \\
        \vspace{3.5pt} \\ 
        $E_1=2.97\times 10^{-2}$\\
        $\delta=-4.03\times 10^{-2}$ \\
        $u\in[-7.01\times 10^{-3}, 1.0366]$ \\
        \includegraphics[scale=0.35]{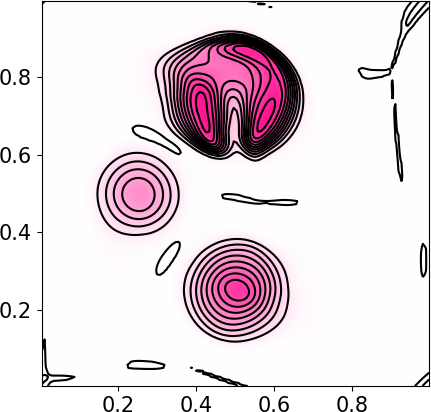}
    \end{tabular}} ~
    \subfloat[FCT-SDIRK]{
      \begin{tabular}{c}
        \includegraphics[scale=0.35]{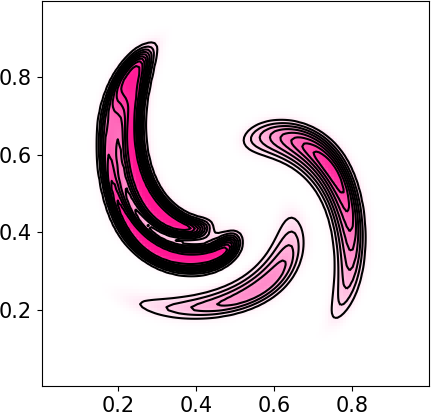} \\
        \vspace{3.5pt} \\ 
        $E_1=2.99\times 10^{-2}$\\
        $\delta=-1.73\times 10^{-18}$ \\
        $u\in[0, 0.9988]$ \\
        \includegraphics[scale=0.35]{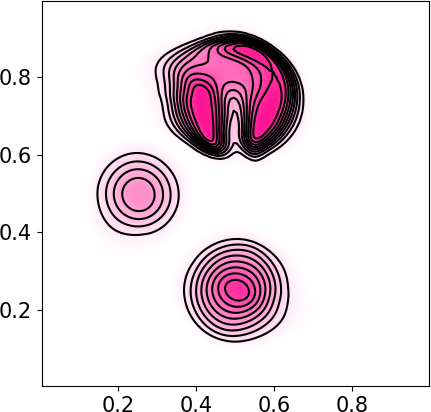}
    \end{tabular}} ~
    \subfloat[GMC-SDIRK]{
      \begin{tabular}{c}
        \includegraphics[scale=0.35]{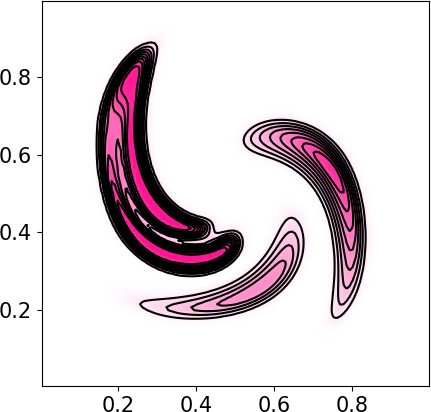} \\
        \vspace{3.5pt} \\ 
        $E_1=2.99\times 10^{-2}$\\
        $\delta=-3.23\times 10^{-19}$ \\
        $u\in[-2.89\times 10^{-26}, 0.9993]$ \\
        \includegraphics[scale=0.35]{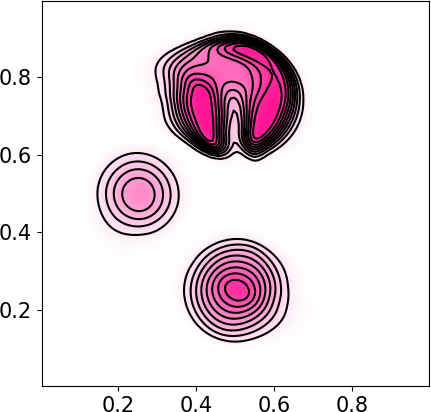}
    \end{tabular}}
  \end{center}
  \caption{Numerical solution of the linear problem \eqref{2D_swirling}
    with initial condition \eqref{3_body_problem} using the different schemes. 
    In all cases, we use $N_h=128^2$ uniform square cells. 
    We show the solution at $t=T/2$ and $t=T=1.5$ in the first and second rows, 
    respectively. The color scale in the plots goes from white to pink, which
    corresponds to $0$ and $1$, respectively. 
    \label{fig:2D_swirling}}
\end{figure}

\subsection{Two-dimensional linear advection diffusion equation}\label{sec:num_2d_linear_convection_diffusion}
Next we solve the linear advection-diffusion equation
\begin{subequations}\label{2D_lin_conv_diff}
\begin{align}
  u_t+u_x+u_y=\epsilon(u_{xx}+u_{yy}), \quad (x,y)\in[0,2\pi]^2,
\end{align}
with $\epsilon=1/1000$.
Following \cite[Example 3.5]{xiong2015high}, 
we impose periodic boundary conditions and take 
\begin{align}
  u(x,y,0) = \sin^4(x+y)
\end{align}
\end{subequations}
as initial condition. The exact solution, also found in \cite{xiong2015high}, is 
\begin{align*}
u(x,y,t)=\frac{3}{8}-\frac{1}{2}\exp(-8\epsilon t)\cos(2(x+y-2t))
+\frac{1}{8}\exp(-32\epsilon t)\cos(4(x+y-2t)).
\end{align*}
We solve the problem up to the final time $t=0.5$ using $\lambda_{ij}^A=1$.
The global bounds are $u^{\min}=0$ and $u^{\max}=1$. The results of a convergence 
study are shown in Table \ref{table:2D_lin_conv_diff}. 
The high-order method WENO-SDIRK produces small 
undershoots and/or overshoots. The rest of the methods produce MPP solutions. 
To recover the high-order accuracy from WENO-SDIRK, we found it necessary to perform 
at least two iterations with FCT-SDIRK and use $\gamma\geq 1$ with GMC-SDIRK.

\begin{table}[!h]\scriptsize
  \begin{center}
    \begin{tabular}{||c||c|c|c||c|c|c||c|c|c||c|c|c||} \hline
      & \multicolumn{3}{c||}{Low-BE} 
      & \multicolumn{3}{c||}{WENO-SDIRK} 
      & \multicolumn{3}{c||}{FCT-SDIRK (2 iter)} 
      & \multicolumn{3}{c||}{GMC-SDIRK ($\gamma=1$)} \\ \hline 
      $\Delta x=\Delta y$ & $E_1$ & rate & $\delta$ & $E_1$ & rate & $\delta$ & $E_1$ & rate & $\delta$  & $E_1$ & rate & $\delta$  \\ \hline
1/25  & 7.72 &   -- & 2.46E-02 & 3.33E-01 &  --  & -4.18E-03 & 3.07E-01 &  --  & -6.94E-18 & 2.88E-01 &  --  & 3.40E-04 \\
1/50  & 5.05 & 0.61 & 3.26E-03 & 4.75E-02 & 2.81 & -5.72E-04 & 5.47E-02 & 2.49 & -4.34E-19 & 4.24E-02 & 2.76 & 1.18E-04 \\
1/100 & 2.98 & 0.76 & 2.73E-04 & 2.63E-03 & 4.17 & -3.11E-05 & 3.06E-03 & 4.16 & -5.42E-20 & 2.47E-03 & 4.10 & 7.44E-06 \\
1/200 & 1.64 & 0.86 & 1.90E-05 & 1.08E-04 & 4.60 & -7.23E-07 & 1.17E-04 & 4.70 & -3.39E-21 & 1.08E-04 & 4.52 & 4.56E-07 \\ \hline
    \end{tabular}
    \caption{Grid convergence study for the linear problem \eqref{2D_lin_conv_diff}. \label{table:2D_lin_conv_diff}}
  \end{center}
\end{table}


\subsection{KPP problem}\label{sec:num_2D_kpp}
We close this work with the two-dimensional nonlinear problem 
\begin{subequations}\label{kpp}
\begin{align}
  u_t + \nabla\cdot \bff(u) = \epsilon \Delta u, \quad (x,y)\in [-2,2]\times[-2.5,1.5]
\end{align}
with a nonconvex flux function
\begin{align}
  \bff(u)=(\sin(u),\cos(u)).
\end{align}
We impose periodic boundary conditions and take 
\begin{align}
  u(x,y,t=0) = 
  \begin{cases}
    \frac{14\pi}{4}, &\mbox{ if } \sqrt{x^2+y^2}\leq 1, \\
    \frac{\pi}{4},   &\mbox{ otherwise }
  \end{cases}
\end{align}
\end{subequations}
as initial condition.
We choose $\epsilon=\{0,0.01\}$. 
When $\epsilon=0$, the problem is known as KPP \cite{kurganov2007adaptive}.
This is a challenging test for verification of preservation of the maximum principle 
and entropy stability properties. The entropy solution contains 
a rotating wave structure, which some numerical methods -- even some first-order methods -- struggle to capture;
see for example \cite[Figure 1]{guermond2016invariant}.
The true solution remains in the interval $[\pi/4, 14\pi/4]$.
In \cite{guermond2014second}, the authors remark that using flux limiters to guarantee 
$u\in[\pi/4, 14\pi/4]$ is not enough to make the method used there converge to the entropy solution. 

In Figure \ref{fig:kpp_low_uBE}, 
we show the solution using Low-BE with $N_h=128^2$ uniform square cells. 
This method is not only MPP, but also entropy stable for 
any entropy pair of \eqref{kpp} with $\epsilon=0$; see e.g. \cite{guermond2016invariant}.
As a result, Low-BE converges to the entropy satisfying solution;
however, the method is excessively dissipative. 
The high-order baseline method WENO-SDIRK, shown in Figure
\ref{fig:kpp_WENO_SDIRK}, delivers sharp fronts but violates the maximum
principle.
In our experiments, WENO-SDIRK is able to reproduce the rotating wave structure of the 
entropy satisfying solution. Both FCT-SDIRK and GMC-SDIRK guarantee the solution is 
within the correct bounds without 
a noticeable degradation in accuracy; see Figures \ref{fig:kpp_FCT_SDIRK} and 
\ref{fig:kpp_GMC_SDIRK}. 

We could improve the robustness of the high-order methods by adding numerical
dissipation of entropy. It is important, however, to do this in a way compatible with 
the rest of the algorithm to still guarantee high-order accuracy and preservation of the maximum principle.
As future work, we plan to combine the methodology in this work with that from
\cite{kuzmin2020algebraic, kuzmin2020entropy}.
Our aim is to achieve an entropy stable and MPP high-order method. 

Finally, in Figure \ref{fig:kpp_eps0p01}, we show the results using $\epsilon=0.01$.

\begin{figure}[!h]
  \begin{center}\scriptsize
    \subfloat[Low-BE\label{fig:kpp_low_uBE}]{
      \begin{tabular}{c}
        $\delta=-1.78\times 10^{-15}$ \\
        $u\in[0.7853, 10.9845]$ \\
        \includegraphics[scale=0.35]{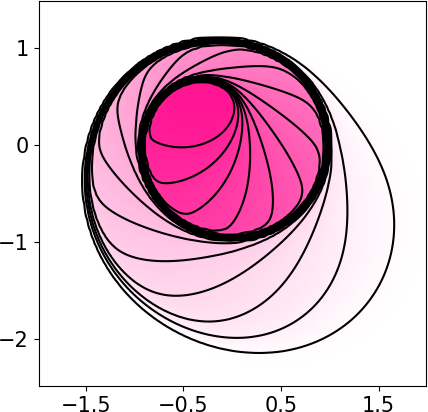} 
      \end{tabular}} ~
    \subfloat[WENO-SDIRK\label{fig:kpp_WENO_SDIRK}]{
      \begin{tabular}{c}
        $\delta=-6.03\times 10^{-8}$ \\
        $u\in[0.7853, 10.9955]$ \\
        \includegraphics[scale=0.35]{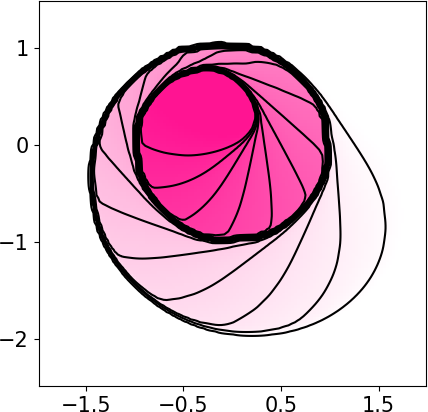} 
    \end{tabular}} ~
    \subfloat[FCT-SDIRK\label{fig:kpp_FCT_SDIRK}]{
      \begin{tabular}{c}
        $\delta=-1.07\times 10^{-14}$ \\
        $u\in[0.7853, 10.9955]$ \\
        \includegraphics[scale=0.35]{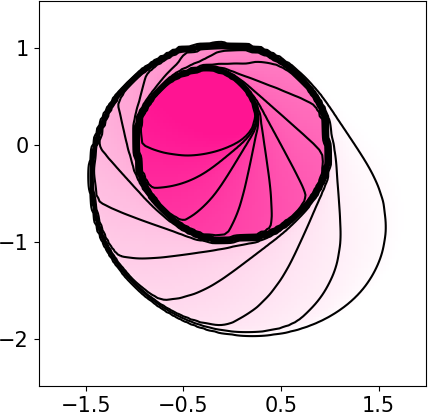} 
    \end{tabular}} ~
    \subfloat[GMC-SDIRK\label{fig:kpp_GMC_SDIRK}]{
      \begin{tabular}{c}
        $\delta=-2.31\times 10^{-14}$ \\
        $u\in[0.7853, 10.9955]$ \\
        \includegraphics[scale=0.35]{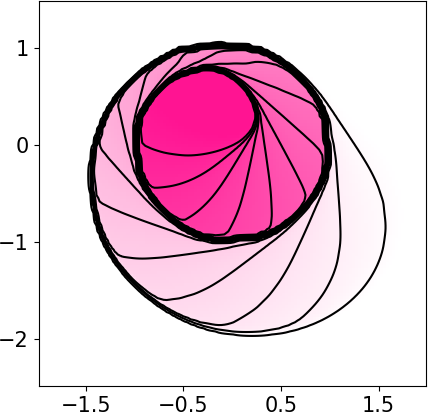} 
    \end{tabular}}
  \end{center}
  \caption{Numerical solution (at $t=1$) of the nonlinear problem \eqref{kpp} with $\epsilon=0$.
    We use different schemes with $N_h=128^2$ uniform square cells. 
    The color scale in the plots goes from white to pink, which
    corresponds to $\pi/4$ and $14\pi/4$, respectively. 
    \label{fig:kpp_eps0}}
\end{figure}

\begin{figure}[!h]
  \begin{center}\scriptsize
    \subfloat[Low-BE]{
      \begin{tabular}{c}
        $\delta=-3.55\times 10^{-15}$ \\
        $u\in[0.7853, 10.9160]$ \\
        \includegraphics[scale=0.35]{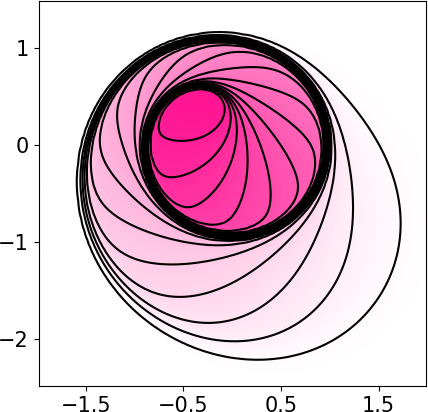} 
      \end{tabular}} ~
    \subfloat[WENO-SDIRK]{
      \begin{tabular}{c}
        $\delta=-5.95\times 10^{-8}$ \\
        $u\in[0.7853, 10.9953]$ \\
        \includegraphics[scale=0.35]{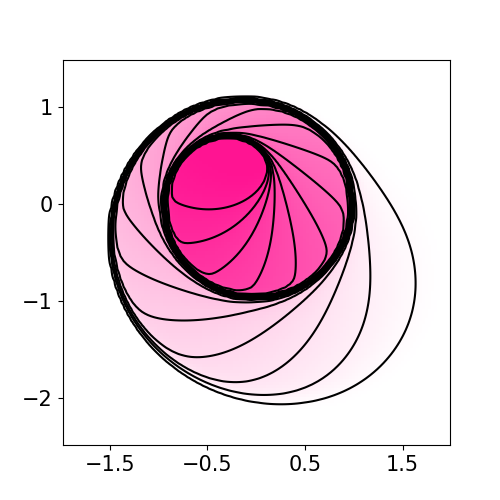} 
    \end{tabular}} ~
    \subfloat[FCT-SDIRK]{
      \begin{tabular}{c}
        $\delta=-1.07\times 10^{-14}$ \\
        $u\in[0.7853, 10.9952]$ \\
        \includegraphics[scale=0.35]{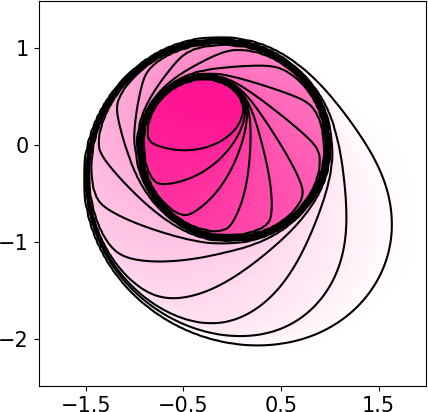} 
    \end{tabular}} ~
    \subfloat[GMC-SDIRK]{
      \begin{tabular}{c}
        $\delta=-1.78\times 10^{-15}$ \\
        $u\in[0.7853, 10.9953]$ \\
        \includegraphics[scale=0.35]{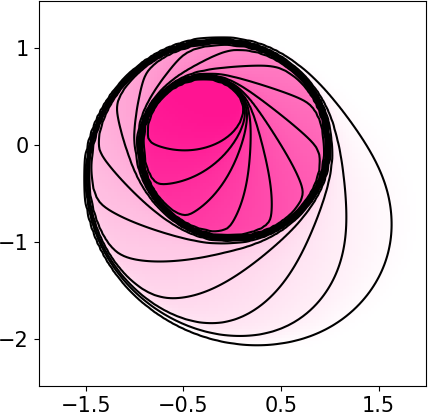} 
    \end{tabular}}
  \end{center}
  \caption{Numerical solution (at $t=1$) of the nonlinear problem \eqref{kpp} with $\epsilon=0.01$.
    We use different schemes with $N_h=128^2$ uniform square cells. 
    The color scale in the plots goes from white to pink, which
    corresponds to $\pi/4$ and $14\pi/4$, respectively. 
    \label{fig:kpp_eps0p01}}
\end{figure}

\section{Conclusions}\label{sec:conclusions}
We have presented two techniques to obtain maximum principle preserving (MPP) numerical
schemes for scalar nonlinear convection-diffusion PDEs, following an approach
similar to that of \cite{exGMCL}, which focused on explicit methods for hyperbolic
problems. Both methodologies 
are based on combining a low-order MPP scheme with a high order scheme, limiting the
contribution from their difference.
While we have focused on using finite volumes in space and Runge-Kutta 
methods in time, the limiters developed here could be used with a wide range 
of space and time discretizations. 
Using these limiters with appropriate discretizations,
one can obtain a scheme whose local error is of any desired order and
use a time step that is restricted only by accuracy considerations. 
That is, the methods are MPP for time steps of any size. 

Since our MPP limiters don't impose a local TVD or non-oscillatory
property on their own, a key ingredient in our 
methodology is to start with a high-order spatial discretization (like WENO) that 
produces only small violations of the maximum principle. 
As an alternative to WENO limiting, one could employ for example
finite element methods that with flux limiters that 
impose local bounds and then relax the constraint around smooth extrema;
see for example \cite{lohmann2017flux, kuzmin2020subcell}.
Since our time discretization method need not be SSP, we avoid the
well-known order barriers to which SSP methods are subject.

In the future, we plan to combine these limiters with the algebraic entropy-stable  
fluxes from \cite{kuzmin2020algebraic, kuzmin2020entropy} to obtain a high-order, 
entropy-stable, and MPP scheme. 
In addition, we plan to apply these limiters to systems of PDEs where bound
preservation is important, such as the compressible Navier-Stokes equations.

\section*{Acknowledgment}
We are grateful to Prof. Dmitri Kuzmin for important discussions that formed
the basis of this work, for providing feedback on drafts of the paper and for suggesting 
the fixed point iteration \eqref{gmc_iter}.

\section*{Declarations}
\subsection*{Funding}
This work was funded by King Abdullah University of Science and Technology (KAUST) in Thuwal, Saudi Arabia.

\subsection*{Conflicts of interest/Competing interests}
The authors declare that they have no known conflicts of interest, competing interests 
or personal relationships that could have appeared to influence the work reported in this paper.

\subsection*{Availability of data and material}
The code to reproduce the datasets (in all tables) is available at 
\url{https://github.com/manuel-quezada/BP\_Lim\_for\_imp\_RK\_Methods}.

\subsection*{Code availability}
The code to reproduce the numerical experiments is available at 
\url{https://github.com/manuel-quezada/BP\_Lim\_for\_imp\_RK\_Methods}.

\bibliographystyle{plain}
\bibliography{refs}

\begin{thebibliography}{10}

\bibitem{anderson2017high}
R~Anderson, Veselin Dobrev, Tz~Kolev, Dmitri Kuzmin, M~Quezada de~Luna,
  R~Rieben, and V~Tomov.
\newblock High-order local maximum principle preserving ({MPP}) discontinuous
  {G}alerkin finite element method for the transport equation.
\newblock {\em Journal of Computational Physics}, 334:102--124, 2017.

\bibitem{arbogast2020third}
Todd Arbogast, Chieh-Sen Huang, Xikai Zhao, and Danielle~N King.
\newblock A third order, implicit, finite volume, adaptive {R}unge--{K}utta
  {WENO} scheme for advection--diffusion equations.
\newblock {\em Computer Methods in Applied Mechanics and Engineering},
  368:113155, 2020.

\bibitem{bolley1978}
Catherine Bolley and Michel Crouzeix.
\newblock {Conservation de la positivit\'{e} lors de la discr\'{e}tisation des
  probl\'{e}mes d'\'{e}volution paraboliques}.
\newblock {\em R.A.I.R.O. Analyse Num\'{e}rique}, 12(3):237--245, 1978.

\bibitem{boris1973flux}
Jay~P Boris and David~L Book.
\newblock Flux-corrected transport. {I}. {SHASTA}, a fluid transport algorithm
  that works.
\newblock {\em Journal of Computational Physics}, 11(1):38--69, 1973.

\bibitem{chen2016third}
Zheng Chen, Hongying Huang, and Jue Yan.
\newblock Third order maximum-principle-satisfying direct discontinuous
  {G}alerkin methods for time dependent convection diffusion equations on
  unstructured triangular meshes.
\newblock {\em Journal of Computational Physics}, 308:198--217, 2016.

\bibitem{feng2019time}
Dianlei Feng, Insa Neuweiler, Udo Nackenhorst, and Thomas Wick.
\newblock A time-space flux-corrected transport finite element formulation for
  solving multi-dimensional advection-diffusion-reaction equations.
\newblock {\em Journal of Computational Physics}, 396:31--53, 2019.

\bibitem{SSPbook}
Sigal Gottlieb, David~I. Ketcheson, and Chi-Wang Shu.
\newblock {\em Strong Stability Preserving {R}unge-{K}utta And Multistep Time
  Discretizations}.
\newblock WORLD SCIENTIFIC, January 2011.

\bibitem{guermond2014second}
Jean-Luc Guermond, Murtazo Nazarov, Bojan Popov, and Yong Yang.
\newblock A second-order maximum principle preserving lagrange finite element
  technique for nonlinear scalar conservation equations.
\newblock {\em SIAM Journal on Numerical Analysis}, 52(4):2163--2182, 2014.

\bibitem{guermond2016invariant}
Jean-Luc Guermond and Bojan Popov.
\newblock Invariant domains and first-order continuous finite element
  approximation for hyperbolic systems.
\newblock {\em SIAM Journal on Numerical Analysis}, 54(4):2466--2489, 2016.

\bibitem{Hairer:ODEs2}
Ernst Hairer and G.~Wanner.
\newblock {\em {Solving ordinary differential equations \{II\}: Stiff and
  differential-algebraic problems}}, volume~14 of {\em Springer Series in
  Computational Mathematics}.
\newblock Springer, second edition, 1996.

\bibitem{hajduk2021monolithic}
Hennes Hajduk.
\newblock Monolithic convex limiting in discontinuous {G}alerkin
  discretizations of hyperbolic conservation laws.
\newblock {\em Computers \& Mathematics with Applications}, 87:120--138, 2021.

\bibitem{harten1972self}
A~Harten and G~Zwas.
\newblock Self-adjusting hybrid schemes for shock computations.
\newblock {\em Journal of Computational Physics}, 9(3):568--583, 1972.

\bibitem{harten1997high}
Ami Harten.
\newblock High resolution schemes for hyperbolic conservation laws.
\newblock {\em Journal of computational physics}, 135(2):260--278, 1997.

\bibitem{harten1974method}
Amiram Harten.
\newblock Method of artificial compression. {I}. shocks and contact
  discontinuities.
\newblock Technical report, New York Univ., NY (USA). AEC Computing and Applied
  Mathematics Center, 1974.

\bibitem{horvath1998positivity}
Zolt{\'a}n Horv{\'a}th.
\newblock Positivity of {R}unge-{K}utta and diagonally split {R}unge-{K}utta
  methods.
\newblock {\em Applied numerical mathematics}, 28(2-4):309--326, 1998.

\bibitem{jameson1993computational}
Antony Jameson.
\newblock Computational algorithms for aerodynamic analysis and design.
\newblock {\em Applied Numerical Mathematics}, 13(5):383--422, 1993.

\bibitem{jiang1996efficient}
Guang-Shan Jiang and Chi-Wang Shu.
\newblock Efficient implementation of weighted {ENO} schemes.
\newblock {\em Journal of computational physics}, 126(1):202--228, 1996.

\bibitem{kennedy2016diagonally}
Christopher~A. Kennedy and Mark~H. Carpenter.
\newblock {\em Diagonally {I}mplicit {R}unge-{K}utta Methods for {O}rdinary
  {D}ifferential {E}quations, a {R}eview}.
\newblock National Aeronautics and Space Administration, Langley Research
  Center, 2016.

\bibitem{2014_hork}
David~I. Ketcheson and Umair {bin Waheed}.
\newblock A comparison of high order explicit {R}unge-{K}utta, extrapolation,
  and deferred correction methods in serial and parallel.
\newblock {\em CAMCoS}, 9(2):175--200, 2014.

\bibitem{kurganov2007adaptive}
Alexander Kurganov, Guergana Petrova, and Bojan Popov.
\newblock Adaptive semidiscrete central-upwind schemes for nonconvex hyperbolic
  conservation laws.
\newblock {\em SIAM Journal on Scientific Computing}, 29(6):2381--2401, 2007.

\bibitem{kuzmin2020monolithic}
Dmitri Kuzmin.
\newblock Monolithic convex limiting for continuous finite element
  discretizations of hyperbolic conservation laws.
\newblock {\em Computer Methods in Applied Mechanics and Engineering},
  361:112804, 2020.

\bibitem{kuzmin2020algebraic}
Dmitri Kuzmin and Manuel~Quezada de~Luna.
\newblock Algebraic entropy fixes and convex limiting for continuous finite
  element discretizations of scalar hyperbolic conservation laws.
\newblock {\em Computer Methods in Applied Mechanics and Engineering},
  372:113370, 2020.

\bibitem{kuzmin2020entropy}
Dmitri Kuzmin and Manuel~Quezada de~Luna.
\newblock Entropy conservation property and entropy stabilization of high-order
  continuous {G}alerkin approximations to scalar conservation laws.
\newblock {\em Computers \& Fluids}, 213:104742, 2020.

\bibitem{kuzmin2020subcell}
Dmitri Kuzmin and Manuel~Quezada de~Luna.
\newblock Subcell flux limiting for high-order {B}ernstein finite element
  discretizations of scalar hyperbolic conservation laws.
\newblock {\em Journal of Computational Physics}, 411:109411, 2020.

\bibitem{exGMCL}
Dmitri Kuzmin, Manuel~Quezada de~Luna, David~I Ketcheson, and Johanna
  Gr{\"u}ll.
\newblock Bound-preserving convex limiting for high-order {R}unge-{K}utta time
  discretizations of hyperbolic conservation laws.
\newblock {\em Preprint: arXiv:2009.01133}, 2020.

\bibitem{kuzmin2012flux}
Dmitri Kuzmin, Rainald L{\"o}hner, and Stefan Turek.
\newblock {\em Flux-corrected transport: principles, algorithms, and
  applications}.
\newblock Springer, 2012.

\bibitem{lee2010multistep}
Jin-Luen Lee, Rainer Bleck, and Alexander~E MacDonald.
\newblock A multistep flux-corrected transport scheme.
\newblock {\em Journal of Computational Physics}, 229(24):9284--9298, 2010.

\bibitem{leveque1992numerical}
Randall~J LeVeque.
\newblock {\em Numerical methods for conservation laws}, volume 132.
\newblock Springer, 1992.

\bibitem{leveque1996high}
Randall~J Leveque.
\newblock High-resolution conservative algorithms for advection in
  incompressible flow.
\newblock {\em SIAM Journal on Numerical Analysis}, 33(2):627--665, 1996.

\bibitem{leveque2002finite}
Randall~J LeVeque.
\newblock {\em Finite volume methods for hyperbolic problems}, volume~31.
\newblock Cambridge University Press, 2002.

\bibitem{liu1994weighted}
Xu-Dong Liu, Stanley Osher, and Tony Chan.
\newblock Weighted essentially non-oscillatory schemes.
\newblock {\em Journal of computational physics}, 115(1):200--212, 1994.

\bibitem{lohmann2017flux}
Christoph Lohmann, Dmitri Kuzmin, John~N Shadid, and Sibusiso Mabuza.
\newblock Flux-corrected transport algorithms for continuous {G}alerkin methods
  based on high order {B}ernstein finite elements.
\newblock {\em Journal of Computational Physics}, 344:151--186, 2017.

\bibitem{magiera2020constraint}
Jim Magiera, Deep Ray, Jan~S Hesthaven, and Christian Rohde.
\newblock Constraint-aware neural networks for riemann problems.
\newblock {\em Journal of Computational Physics}, 409:109345, 2020.

\bibitem{nikitin2014monotone}
Kirill Nikitin, Kirill Terekhov, and Yuri Vassilevski.
\newblock A monotone nonlinear finite volume method for diffusion equations and
  multiphase flows.
\newblock {\em Computational Geosciences}, 18(3-4):311--324, 2014.

\bibitem{osher1984high}
Stanley Osher and Sukumar Chakravarthy.
\newblock High resolution schemes and the entropy condition.
\newblock {\em SIAM Journal on Numerical Analysis}, 21(5):955--984, 1984.

\bibitem{qiu2002construction}
Jianxian Qiu and Chi-Wang Shu.
\newblock On the construction, comparison, and local characteristic
  decomposition for high-order central {WENO} schemes.
\newblock {\em Journal of Computational Physics}, 183(1):187--209, 2002.

\bibitem{rusanov1961calculation}
Viktor~Vladimirovich Rusanov.
\newblock The calculation of the interaction of non-stationary shock waves with
  \ barriers.
\newblock {\em Zhurnal Vychislitel'noi Matematiki i Matematicheskoi Fiziki},
  1(2):267--279, 1961.

\bibitem{spijker1983}
M.~N. Spijker.
\newblock {Contractivity in the numerical solution of initial value problems}.
\newblock {\em Numerische Mathematik}, 42:271--290, 1983.

\bibitem{xiong2015high}
Tao Xiong, Jing-Mei Qiu, and Zhengfu Xu.
\newblock High order maximum-principle-preserving discontinuous {G}alerkin
  method for convection-diffusion equations.
\newblock {\em SIAM Journal on Scientific Computing}, 37(2):A583--A608, 2015.

\bibitem{yang2016high}
Pei Yang, Tao Xiong, Jing-Mei Qiu, and Zhengfu Xu.
\newblock High order maximum principle preserving finite volume method for
  convection dominated problems.
\newblock {\em Journal of Scientific Computing}, 67(2):795--820, 2016.

\bibitem{zalesak1979fully}
Steven~T Zalesak.
\newblock Fully multidimensional flux-corrected transport algorithms for
  fluids.
\newblock {\em Journal of Computational Physics}, 31(3):335--362, 1979.

\bibitem{zhang2012maximum}
Xiangxiong Zhang, Yuanyuan Liu, and Chi-Wang Shu.
\newblock Maximum-principle-satisfying high order finite volume weighted
  essentially nonoscillatory schemes for convection-diffusion equations.
\newblock {\em SIAM Journal on Scientific Computing}, 34(2):A627--A658, 2012.

\bibitem{zhang2010maximum}
Xiangxiong Zhang and Chi-Wang Shu.
\newblock On maximum-principle-satisfying high order schemes for scalar
  conservation laws.
\newblock {\em Journal of Computational Physics}, 229(9):3091--3120, 2010.

\bibitem{zhang2011maximum}
Xiangxiong Zhang and Chi-Wang Shu.
\newblock Maximum-principle-satisfying and positivity-preserving high-order
  schemes for conservation laws: survey and new developments.
\newblock {\em Proceedings of the Royal Society A: Mathematical, Physical and
  Engineering Science}, 467(2134):2752--2776, 2011.

\end{thebibliography}

\newpage
\appendix
\section{Pseudo-Jacobians for Newton-like methods}\label{sec:jacobians}
To use either the FCT or the GMC limiters, first we need to solve $M$ (non)linear systems to 
obtain the high-order fluxes. The FCT method requires solving an extra system
to obtain the low-order solution and its fluxes. In contrast, to use the GMC limiters we do not 
need to obtain the low-order solution, but we need to perform the iterative algorithm \eqref{gmc_iter}. 
An efficient solver for these systems is out of the scope of this work. 
However, for completeness, we describe in this section how we solve these (non)linear systems.

\subsection{Newton-type method for the low-order baseline scheme}
The low-order discretization, given by \eqref{low-order-scheme_bar_states} or \eqref{low-order-scheme_fluxes},
can be solved via Newton's method by defining a residual and its Jacobian. 
Let 
\begin{align*}
  r_i^{L,(k)}=u_i^{L, (k)}-u_i^n+\frac{\Delta t}{|K_i|}\sum_{l\in \N_i}|S_{il}|G_{il}^{L}\left(u^{L,(k)}\right)=0,
  \qquad
  J_{ij}^{L,(k)}= \frac{\partial r_i^{L,(k)}}{\partial u_j^{L,(k)}},
\end{align*}
be the entries of the residual and the Jacobian (evaluated at the $k$-th Newton iteration), respectively.
The corresponding iterative algorithm is 
\begin{align}\label{low_order_newton}
  J^{L,(k)}\left(u^{L,(k+1)}-u^{L,(k)}\right) = - r^{L,(k)}.
\end{align}
To simplify the computation of the Jacobian, we ignore the dependence of $\lambda_{ij}^A$ 
with respect to the solution.
The entries of the (pseudo) Jacobian are
\begin{align}\label{jac_low_order}
  J_{ij}^{L,(k)} = \frac{\partial r_i^{L,(k)}}{\partial u_j^{L,(k)}}
  =\delta_{ij}+\frac{\Delta t}{|K_i|}\sum_{l\in \N_i}|S_{il}|\frac{\partial G_{il}^{L}\left(u^{L,(k)}\right)}{\partial u_j^{L,(k)}},
\end{align}
where $\delta_{ij}$ is the Kronecker delta function. 
For the one-dimensional problem, 
\begin{align*}
  \sum_{l\in \N_i} |S_{il}|\frac{\partial G_{il}^{L}\left(u^{L,(k)}\right)}{\partial u_j^{L,(k)}} 
  &= 
  \frac{\partial \left[F_{i+1/2}^{L,(k)}-F_{i-1/2}^{L,(k)}
  -\left(P_{i+1/2}^{L,(k)}-P_{i-1/2}^{L,(k)}\right)\right]}{\partial u_j^{L,(k)}}=:A_{ij}\left(u^{L,(k)}\right),
\end{align*}
where $F_{i+1/2}^{L,(k)}=F_{ii+1}^L\left(u^{L,(k)}\right)$ 
and $P_{i+1/2}^{L,(k)}=P_{ii+1}^L\left(u^{L,(k)}\right)$
are the low-order fluxes, given by \eqref{conv-flux_low_order} and \eqref{diff-flux_low_order}, respectively.
Using \eqref{conv-flux_low_order} and \eqref{diff-flux_low_order}, we get
\begin{align*}
  A_{ij}(u)
  =\begin{cases}
  -\frac{1}{2}\bff^\prime(u_{i-1})-\frac{\lambda^A_{i-1/2}}{2}
  -\frac{1}{\Delta x}c_{i-1/2}+\frac{1}{2}c^\prime_{i-1/2}\left(\frac{u_i-u_{i-1}}{\Delta x}\right)
  , &\mbox{ if } j=i-1, \\
  \frac{\lambda^A_{i+1/2}}{2}+\frac{\lambda^A_{i-1/2}}{2}
  +\frac{1}{\Delta x}c_{i+1/2}-\frac{1}{2}c^\prime_{i+1/2}\left(\frac{u_{i+1}-u_i}{\Delta x}\right)
  +\frac{1}{\Delta x}c_{i-1/2}+\frac{1}{2}c^\prime_{i-1/2}\left(\frac{u_{i}-u_{i-1}}{\Delta x}\right)
  , &\mbox{ if } j=i, \\
  \frac{1}{2}\bff^\prime(u_{i+1})-\frac{\lambda^A_{i+1/2}}{2} 
  - \frac{1}{\Delta x}c_{i+1/2} - \frac{1}{2}c^\prime_{i+1/2}\left(\frac{u_{i+1}-u_{i}}{\Delta x}\right)
  , &\mbox{ if } j=i+1, \\
  0, &\mbox{ otherwise}.
  \end{cases}
\end{align*}
We run the iterative algorithm \eqref{low_order_newton} until
\begin{align*}
  \left|\left|r^{L,(k+1)}\right|\right|_{\ell^2}\leq \text{tol}^{L}=10^{-12}.
\end{align*}

\subsection{Newton-type method for the high-order baseline scheme}
For the high-order full discretization \eqref{high-order-scheme}, which is based on 
a DIRK method with $M$ stages, we need the $M$ intermediate solutions $y^{(m)}$, given by 
\eqref{interm_RK_soln}. Each of these intermediate solutions can be solved via Newton's method.
Let $y^{(m,k)}$ denote the $k$-th Newton iteration of the intermediate solution $y^{(m)}$. 
Then, 
\begin{align*}
  r_i^{\RK,(m,k)} &= y_i^{(m,k)}-u_i^n+\frac{\Delta t}{|K_i|}\sum_{l\in \N_i}|S_{il}|\sum_{s=1}^m a_{ms}
  \Bigg[
    \underbrace{F_{il}^{H}\left(y^{(s,k)},\bfx_{ij}\right)-P_{il}^{H}\left(y^{(s,k)},\bfx_{ij}\right)}_
               {\textstyle=:G^{\RK}_{il}\left(y^{(s,k)},\bfx_{ij}\right)}
    \Bigg], \\
  J_{ij}^{\RK,(m,k)} &= \frac{\partial r_i^{\RK,(m,k)}}{\partial y_j^{(m,k)}}
\end{align*} 
are the entries of the residual and the Jacobian (evaluated at the $k$-th Newton iteration), respectively.
The iterative algorithm to solve for the $m$-th intermediate solution is 
\begin{align}\label{rk_stages_newton}
  J^{\RK,(m,k)}\left(y^{(m,k+1)}-y^{(m,k)}\right) = -r^{\RK,(m,k)}.
\end{align}
The entries of the Jacobian are 
\begin{align*}
  J_{ij}^{\RK,(m,k)} = \frac{\partial r_i^{\RK,(m,k)}}{\partial y_j^{(m,k)}}
  &=\delta_{ij}
  +\frac{\Delta t}{|K_i|}\sum_{l\in \N_i}|S_{il}|
  \Bigg[
    \sum_{s=1}^{m-1} a_{ms}\underbrace{\frac{\partial G_{il}^{\RK}\left(y^{(s,k)},\bfx_{ij}\right)}{\partial y_j^{(m,k)}}}_{=0}
    +a_{mm}\frac{\partial G_{il}^{\RK}\left(y^{(m,k)},\bfx_{ij}\right)}{\partial y_j^{(m,k)}}
    \Bigg] \\
  &=\delta_{ij}+\frac{a_{mm}\Delta t}{|K_i|}\sum_{l\in \N_i}|S_{il}|\frac{\partial G_{il}^{\RK}\left(y^{(m,k)},\bfx_{ij}\right)}{\partial y_j^{(m,k)}}.
\end{align*}
Due to the highly nonlinear nature of WENO schemes, the computation of 
$\partial G_{il}^{\RK}\left(y^{(m,k)},\bfx_{ij}\right)/\partial y_j^{(m,k)}$ is complicated. 
Instead, we consider 
\begin{align}\label{jac_high_order}
  J_{ij}^{\RK,(m,k)} \approx 
  \delta_{ij}+\frac{a_{mm}\Delta t}{|K_i|}\sum_{l\in \N_i}|S_{il}|\frac{\partial G_{il}^{L}\left(y^{(m,k)}\right)}{\partial y_j^{(m,k)}}
\end{align}
and ignore the dependence of $\lambda^A_{ij}$ with respect to the solution. 
We run the iterative algorithm \eqref{rk_stages_newton} until 
\begin{align*}
  \left|\left|r^{\RK,(m,k+1)}\right|\right|_{\ell^2} \leq \text{tol}^{\RK} = 10^{-8}.
\end{align*}

\subsection{Pseudo-Jacobian based on linear convection-diffusion problem}\label{sec:linJac}
Simple, non-expensive but potentially inaccurate pseudo-Jacobians can be computed 
based on a linearization of \eqref{conslaw}. That is, considering
\begin{align*}
  u_t+ \mathbf{f}^\prime(\bar u) \cdot \nabla u = c(\bar u) \Delta u,
\end{align*}
where $\bar u \in[\min _\bfx u(\bfx,0), \max_\bfx u(\bfx,0)]$ is a constant based on the initial data. 
For instance, in the numerical experiments of Section \ref{sec:num}, we use 
$\bar u=\frac{1}{2}\left[\max (u(x,0)) - \min(u(x,0))\right]$. 
By doing this, we can pre-compute the factors of the Jacobian (e.g., using an LU decomposition)
to avoid recomputing the Jacobian and solving systems at every time step. The disadvantage
of this approach is that the number of Newton iterations might increase considerably.

\end{document}